\documentclass[11pt,reqno]{amsproc}

\usepackage[foot]{amsaddr}
\usepackage{amsthm,amsmath,amsfonts,amssymb,mathtools,mathrsfs,xparse,bbm,graphicx,mathpazo}
\usepackage[all,cmtip]{xy}
\usepackage[margin=.85in]{geometry}
\allowdisplaybreaks
\renewcommand{\mathcal}{\mathscr}

\usepackage[
backend=bibtex,
style=numeric,
sorting=nyt,
]{biblatex}
\usepackage{hyperref}
\hypersetup
{
	colorlinks,
	linkcolor = blue,
	citecolor = blue,
	urlcolor = blue,
	filecolor = blue,
	breaklinks,
}

\numberwithin{equation}{section}

\theoremstyle{plain}
\newtheorem{theorem}{Theorem}[section]
\newtheorem{proposition}[theorem]{Proposition}
\newtheorem{lemma}[theorem]{Lemma}
\newtheorem{corollary}[theorem]{Corollary}
\newtheorem{claim}[theorem]{Claim}

\theoremstyle{remark}
\newtheorem{remark}[theorem]{Remark}
\newtheorem{example}[theorem]{Example}
\newtheorem{definition}[theorem]{Definition}
\newtheorem{assumption}[theorem]{Assumption}

\newcommand{\RR}{\mathbb{R}}

\newcommand{\set}[1]{{\left\{{#1}\right\}}}
\newcommand{\setc}[2]{\left\{\,#1 \;\vert\; #2\,\right\}}

\newcommand{\pow}[1]{\mathcal{P}(#1)}
\newcommand{\abs}[1]{\left| #1 \right|}
\newcommand{\norm}[1]{\left\lVert#1\right\rVert}
\newcommand{\linfnorm}[1]{\norm{#1}_{\infty}}

\newcommand{\PP}{\mathbb{P}}
\newcommand{\EE}{\mathbb{E}}
\newcommand{\EEbracket}[1]{\EE\left\{ #1 \right\}}
\DeclareMathOperator{\Var}{Var}
\DeclareMathOperator{\Cov}{Cov}
\newcommand{\ind}[1]{\mathbbm{1}\left(#1\right)}

\newcommand{\mvn}[2]{\mathrm{MVN}(#1, #2)}

\newcommand{\lk}{\mathbf{lk}}

\newcommand{\cplx}{\mathcal{L}}

\NewDocumentCommand{\gnp}{O{n} O{p}}{{\textbf G}(#1,#2)}
\NewDocumentCommand{\xnp}{O{n} O{p}}{{\textbf X}(#1,#2)}

\newcommand{\I}{\mathbb{I}}
\newcommand{\D}{\mathbb{D}}
\newcommand{\testzero}{\mathcal{H}_d}

\title[Multivariate Central Limit Theorems for Random Clique Complexes]{Multivariate Central Limit Theorems for Random Clique Complexes}
\author{Tadas Tem\v{c}inas}\address[TT]{Department of Statistics, University of Oxford}\email{tadas.temcinas@keble.ox.ac.uk}
\author{Vidit Nanda}\address[VN]{Mathematical Institute, University of Oxford}\email{nanda@maths.ox.ac.uk}
\author{Gesine Reinert}\address[GR]{Department of Statistics, University of Oxford, and the Alan Turing Institute, London}\email{reinert@stats.ox.ac.uk}

\begin{document}
	\maketitle
	
	\begin{abstract}
		Motivated by open problems in applied and computational algebraic topology, we establish multivariate normal approximation theorems for three random vectors which arise organically in the study of random clique complexes. These are:
		\begin{enumerate}
		\item the vector of critical simplex counts attained by a lexicographical Morse matching,
		\item the vector of simplex counts in the link of a fixed simplex, and 
		\item the vector of total simplex counts.
		\end{enumerate}
		The first of these random vectors forms a cornerstone of modern homology algorithms, while the second one provides a natural generalisation for the notion of vertex degree, and the third one may be viewed from the perspective of $U$-statistics. To obtain distributional approximations for these random vectors, we extend the notion of  dissociated sums to a multivariate setting and prove a new central limit theorem for such sums using Stein's method.
		
		\medskip
		
		
		\medskip
		
	\end{abstract}
	
	\section{Introduction}
	
   Methods from applied and computational algebraic topology have recently found substantial applications in the analysis of nonlinear and unstructured datasets \cite{ghrist:08, tda2}. The modus operandi of topological data analysis is to first build a nested family of simplicial complexes around the elements of a dataset, and to then compute the associated persistent homology barcodes \cite{edelsbrunner_comp_top}. Of central interest, when testing hypotheses under this paradigm, is the question of what homology groups to expect when the input data are randomly generated. Significant efforts have therefore been devoted to answering this question for various models of noise, giving rise to the field of {\em stochastic topology} \cite{kahle2011random, bobrowski2018topology, kahle2009topology, crackle, costa2016large}. Our work here is a contribution to this area at the interface between probability theory and algebraic topology.

A cornerstone for statistical inference, beyond providing expectations, are distributional approximations.	This paper establishes the first multivariate normal approximations to three important counting problems in stochastic topology; as these approximations are based on Stein's method, explicit bounds on the approximation errors are provided.
	Our starting point is the ubiquitous  graph model $\gnp$; a graph $G$ chosen from this model has as its vertex set $[n] = \set{1,2,\ldots,n}$, and each of its possible $\binom{n}{2}$ edges is included independently with probability $p \in [0,1]$. It was established by Erd\H{o}s and R\'enyi in \cite{erdos1959random} that $p = \log(n)/n$ is a {\em sharp threshold} for connectivity in $\gnp$, in the sense that the following assertions hold for any random graph $G \sim \gnp$ and every arbitrarily small $\epsilon > 0$: if $p$ exceeds $(1+\epsilon)\cdot{\log(n)}/{n}$, then $G$ is connected with high probability. Conversely, if $p$ is smaller than $(1-\epsilon)\cdot{\log(n)}/{n}$, then $G$ is disconnected with high probability. 
	
	A natural higher-order generalisation of $\gnp$ is furnished by the random clique complex model $\xnp$, whose constituent complexes $\cplx$ are constructed as follows. One first selects an underlying graph $G \sim \gnp$, and then deterministically fills out all $k$-cliques in $G$ with $(k-1)$-dimensional simplices for $k \geq 3$. Higher connectivity is now measured by the Betti numbers $\beta_k(\cplx)$, which are ranks of rational homology groups $\text{H}_k(\cplx;\mathbb{Q})$ --- in particular, $\beta_0(\cplx)$ equals the number of connected components of the underlying random graph $G$. In \cite{kahle2014sharp}, Kahle proved the following far-reaching generalisation of the Erd\H{o}s-R\'enyi connectivity result: for each $k \geq 1$ and $\epsilon > 0$,
	\begin{enumerate}
		\item if 
		\[
		p \geq \left[\left(\frac{k}{2}+1+\epsilon\right) \cdot \frac{\log(n)}{n}\right]^{{1}/{(k+1)}},
		\] 
		then $\beta_k(\cplx) = 0$ with high probability; and moreover,
		\item if 
		\[
		\left[\frac{k+1+\epsilon}{n}\right]^{1/k} \leq p \leq \left[\left(\frac{k}{2}+1-\epsilon\right) \cdot\frac{\log(n)}{n}\right]^{1/(k+1)},
		\]
		then $\beta_k(\cplx) \neq 0$ with high probability.
	\end{enumerate}

	With this result in mind, we motivate and describe three random vectors pertaining to $\cplx \sim \xnp$; the normal approximation of these three random vectors will be our focus in this paper. All three are denoted $T = (T_1,\ldots,T_d)$ for an integer $d > 0$.
	 
		\subsection*{Random Vector 1: Critical Simplex Counts}
			
	The computation of Betti numbers $\beta_k(\cplx)$ begins with the chain complex	
	\[
	\xymatrixcolsep{.55in}
	\xymatrix{
		\cdots \ar@{->}[r]^{d_{k+1}} & C_k \ar@{->}[r]^{d_k} & C_{k-1} \ar@{->}[r]^{d_{k-1}} & \cdots \ar@{->}[r]^{d_2} & C_1 \ar@{->}[r]^{d_1} & C_0. 
	}
	\] 
	Here $C_k$ is a vector space whose dimension equals the number of $k$-simplices in  $\cplx$, while $d_k:C_k \to C_{k-1}$ is an incidence matrix encoding which $(k-1)$-simplices lie in the boundary of a given $k$-simplex. These matrices satisfy the property that every successive composite $d_{k+1} \circ d_k$ equals zero, and $\beta_k(\cplx)$ is the dimension of the quotient vector space $\ker d_k/\text{img }d_{k+1}$. Thus, one is required to diagonalise the matrices $\set{d_k:C_k \to C_{k-1}}$ via row and column operations, which is a straightforward task in principle. Unfortunately, Gaussian elimination on an $m \times m$ matrix incurs an $O(m^3)$ cost, which becomes prohibitive when facing simplicial complexes built around large data sets \cite{roadmap_persistence}. The standard remedy is to construct a much smaller chain complex which has the same homology groups, and by far the most fruitful mechanism for achieving such homology-preserving reductions is {\bf discrete Morse theory} \cite{forman2002user, mischaikow2013morse, henselman, lampret}. 
	
	The key structure here is that of an {\em acyclic partial matching}, which pairs together certain adjacent simplices of $\cplx$; and the homology groups of $\cplx$ may be recovered from a chain complex whose vector spaces are spanned by unpaired, or {\em critical}, simplices. One naturally seeks an {optimal} acyclic partial matching on $\cplx$ which admits the fewest possible critical simplices. Unfortunately, the optimal matching problem is computationally intractable to solve \cite{joswig} even approximately \cite{bauer} for large $\cplx$. 	Our third random vector is obtained by letting $T_k$ equal the number of critical $k$-simplices for a specific type of acyclic partial matching on $\cplx$, called the {\em lexicographical} matching. Knowledge of this random vector serves to simultaneously quantify the benefit of using discrete Morse theoretic reductions on random simplical complexes and to provide a robust null model by which to measure their efficacy on general (i.e., not necessarily random) simplicial complexes.

	\subsection*{Random Vector 2: Link Simplex Counts}
	
	The {\em link} of a simplex $t$ in $\cplx$, denoted $\lk(t)$,  consists of all simplices $s$ for which the union $s \cup t$ is also a simplex in $\cplx$ and the intersection $s \cap t$ is empty. The link of $t$ forms a simplicial complex in its own right; and if we restrict attention to the underlying random graph $G$, then the link of a vertex is precisely the collection of its neighbours. Therefore, the Betti numbers $\beta_k(\lk(t))$ generalise the degree distribution for vertices of random graphs in two different ways --- one can study neighbourhoods of higher-dimensional simplices by increasing the dimension of $t$, and one can examine higher-order connectivity properties by increasing the homological dimension $k$. The second random vector of interest to us here is obtained by letting $T_k$ equal the number of $k$-simplices lying in the link of a fixed simplex $t$ in $\cplx$,  given that $t$ indeed is a simplex in the random complex. As far as we are aware, ours is the first work that studies this random vector. A different conditional distribution, which follows directly from results on subgraph counts in $\gnp$, has been studied before, see Remark \ref{remark:link_counts}.
	
    There are compelling reasons to better understand the combinatorics and topology of such links from a probabilistic viewpoint. For instance, the fact that the link of a $k$-simplex in a triangulated $n$-manifold is always a triangulated sphere of dimension $(n-k-1)$ has been exploited to produce canonical stratifications of simplicial complexes into homology manifolds \cite{strat2, strat1}. Knowledge of simplex counts (and hence, Betti numbers) of links would therefore form an essential first step in any systematic study involving canonical stratifications of random clique complexes. 
	
    \subsection*{Random Vector 3: Total Simplex Counts}
		
	The strategy employed in Kahle's proof of the second assertion above involves first checking that the expected number of $k$-simplices in $\cplx \sim \xnp$ is much larger than the expected number of simplices of dimensions $k \pm 1$ whenever $p$ lies in the range indicated by (2). Therefore, one may combine the Morse inequalities with the linearity of expectation in order to guarantee that the expected $\beta_k(\cplx)$ is nonzero --- see \cite[Section 4]{kahle2014sharp} for details. To  facilitate more refined analysis and estimates of this sort, the first random vector we study in this paper is obtained by letting $T_k$ equal the total number of $k$-dimensional simplices in $\cplx$.
	
	Since $T_k$ is precisely the number of $(k+1)$-cliques in $G \sim \gnp$, this random vector falls within the purview of {\em generalised $U$-statistics}. We extend results from \cite{janson1991asymptotic} 
	to show not only distributional convergence asymptotically but a stronger result, detailing explicit non-asymptotic bounds on the approximation. Several interesting problems can be seen as special cases --- these include classical $U$-statistics \cite{lee2019u, korolyuk2013theory}, monochromatic subgraph counts of inhomogeneous random graphs with independent random vertex colours, and the number of overlapping patterns in a sequence of independent Bernoulli trials. To the best of our knowledge, this is the first multivariate normal approximation result with explicit bounds where the sizes of the subgraphs are permitted to increase with $n$.

	\subsection*{Main Results} 
	
	The central contributions of this work are multivariate normal approximations for all three random vectors $T$ described above. For the purposes of these introductory remarks, we will restrict attention to the case where $T$ is the vector of critical simplex counts. Letting $\{Y_{i,j}\}_{1 \leq i < j \leq n}$ be a sequence of i.i.d. Bernoulli variables, the $k$-th component is 
	\[
	T_k = \sum_{\substack{s \subset [n] \\ |s|=k}}~ \prod_{i \neq j \in s} Y_{i,j} \left[ \prod_{i=1}^{\min(s) - 1}\left(1 - \prod_{j \in s}Y_{i,j}\right) - \prod_{i=1}^{\min(s) - 1}\left(  1 - \prod_{j \in s_{-}} Y_{i,j}\right)  \right].
	\]
	This variable, which we discuss in Section \ref{section:crit_lexi}, arises naturally in stochastic topology \cite{bobrowski2022random, bobrowski2018topology} but has been poorly studied from a distributional approximation perspective. To the best of our knowledge, only the expected value of a closely-related random variable has been calculated (see \cite[Section 8]{bauer2021parameterized}). While there is no shortage of multivariate normal approximation theorems \cite{fang2016multivariate, raivc2004multivariate, meckes2009stein, chen2011normal}, the existing ones are not sufficiently fine-grained for our purposes. We therefore return to the pioneering work of Barbour, Karo\'{n}ski, and Ruci\'{n}ski \cite{barbour1989central}, who proved a univariate central limit theorem (CLT) for a decomposable sum of random variables using Stein's method, treating the case of dissociated sums as a special case. Our approximation result, described below, forms a new extension of their ideas to the multivariate setting, and may be of independent interest.

	Let $n$ and $d$ be positive integers. For each $i \in [d] =:\{1, 2, \ldots, d\}$, we fix an index set $\I_i \subset [n] \times \set{i}$ and consider the union of disjoint sets $\I =: \bigcup_{i \in [d]} \I_i$. Associate to each such $s  = (k,i) \in \I$ a real centered  random variable $X_s$ and form for each $i \in [d]$ the sum
	\[
	W_i \coloneqq \sum_{s \in \I_i} X_s.
	\] 
	Consider the resulting random vector $W = (W_1,\ldots,W_d) \in \RR^d$. The following notion is a natural multivariate generalisation of the dissociated sum from \cite{moginley1975dissociated}; see also \cite{barbour1989central}.
	
	\begin{definition}\label{def:disrand}
		We call $W$ a {\bf vector of dissociated sums} if for each $s \in \I$ and $j \in [d]$ there exists a {\em dependency neighbourhood} $\D_j(s) \subset \I_j$ satisfying three criteria:
		\begin{enumerate}
			\item the difference $\left(W_j - \sum_{u \in \D_j(s)} X_u\right)$ is independent of $X_s$; 
			\item for each $t \in \I$, the quantity $\left(W_j - \sum_{u \in \D_j(s)} X_u - \sum_{v \in \D_j(t) \setminus \D_j(s)} X_v\right)$ is independent of the pair $(X_s,X_t)$; and finally,
			\item $X_s$ and $X_t$ are independent if $t \not\in \bigcup_j \D_j(s)$. 
		\end{enumerate}
	\end{definition}
	
	Let $W$ be a vector of dissociated sums as defined above.  For each $s\in \I$, by construction, the sets  $\D_j(s), j \in [d]$ are disjoint (although for $s \ne t$, the sets $\D_j(s)$ and $\D_j(t)$ may not be disjoint). We write $\D(s) = \bigcup_{j \in [d]}\D_j(s)$ for the disjoint union of these of dependency neighbourhoods. With this preamble in place, we state our main result.

	\begin{theorem}\label{theorem:mvn_dissociated_decomp_approx}
		Let $h: \RR^d \to \RR$ be any three times continuously differentiable function whose third partial derivatives are Lipschitz continuous and bounded. Consider a standard $d$-dimensional Gaussian vector $Z \sim \mvn{0}{\text{\rm Id}_{d \times d}}$. Assume that for all $s \in \I$, we have $\EEbracket{X_s} = 0$ and $\EE \abs {X_s^3} < \infty.$ Then, for any vector of dissociated sums $W \in \RR^d$ with a positive semi-definite covariance matrix $\Sigma$,
		\[
		\abs{\EE h(W) - \EE h(\Sigma^{\frac{1}{2}}Z)} \leq B_{\ref{theorem:mvn_dissociated_decomp_approx}} \sup_{i, j, k \in [d]} \linfnorm{\frac{\partial^3 h}{\partial x_i \partial x_j \partial x_k}},
		\]
		where $B_{\ref{theorem:mvn_dissociated_decomp_approx}} = B_{\ref{theorem:mvn_dissociated_decomp_approx}.1} + B_{\ref{theorem:mvn_dissociated_decomp_approx}.2}$ is the sum given by
		\begin{align*}
		B_{\ref{theorem:mvn_dissociated_decomp_approx}.1} &\coloneqq 
		\frac{1}{3}\sum_{s \in \I} \sum_{t, u \in \D(s)} \hspace{-.4em} \left(\frac{1}{2}\EE\abs{X_sX_tX_u}+\EE\abs{X_sX_t} \EE\abs{X_u}\right) \\
		B_{\ref{theorem:mvn_dissociated_decomp_approx}.2} &\coloneqq  \frac{1}{3} \sum_{s \in \I} \sum_{t \in \D(s)} \sum_{v \in \D(t) \setminus \D(s)} \hspace{-.4em} \left(\EE\abs{X_sX_tX_v} + \EE\abs{X_sX_t} \EE\abs{X_v}\right). 
		\end{align*}
	\end{theorem}

	In the special case where each $W_k$ is a sum of an equal number of i.i.d. random variables and each i.i.d. sequence is independent, the bound in Theorem \ref{theorem:mvn_dissociated_decomp_approx} is optimal with respect to the size $n$ of the sum in each component. However, compared to the CLT from \cite{fang2016multivariate}, the bound is not optimal in the length $d$ of the vector $W$. In any event, the desired CLT for critical simplex counts follows as a corollary to Theorem \ref{theorem:mvn_dissociated_decomp_approx}. We state a simplified version of this result here and note that the full statement and proof have been recorded as Theorem \ref{theorem:crit_lexi_approx} below. In the statement below, $W \in \RR^d$ is an appropriately scaled and centered vector whose $k$-th component counts the number of critical simplices of dimension $k$ for the lexicographical acyclic partial matching on $\cplx \sim \xnp$.
	
	\begin{theorem}\label{theorem:crit_lexi_approx_intro}
		Let $Z \sim \mvn{0}{\text{\rm Id}_{d \times d}}$ and $\Sigma$ be the covariance matrix of $W$.
		Let $h: \RR^d \to \RR$ be three times partially differentiable function whose third partial derivatives are Lipschitz continuous and bounded. Then there is a constant $B_{\ref{theorem:crit_lexi_approx_intro}} >0$ independent of $n$ and a natural number $N_{\ref{theorem:crit_lexi_approx_intro}}$ such that for any $n \geq N_{\ref{theorem:crit_lexi_approx_intro}}$ we have
		\[\abs{\EE h(W) - \EE h(\Sigma^{\frac{1}{2}}Z)} \leq B_{\ref{theorem:crit_lexi_approx_intro}} \sup_{i, j, k \in [d]} \linfnorm{\frac{\partial^3 h}{\partial x_i \partial x_j \partial x_k}} n^{-1}.\]
	\end{theorem}
	
	En route to proving Theorem \ref{theorem:crit_lexi_approx_intro}, we also establish the following properties, which are of direct interest in computational topology. Here we assume that $p \in (0,1)$ and $k \in \{1, 2, \ldots\}$ are constants.
	\begin{enumerate}
		\item The expected number of critical $k$-simplices is one order of $n$ smaller that the expected number of total $k$-simplices; see Lemma \ref{lemma:expect_crit_lexi}.
		\item The variance of the number of critical $k$-simplices is at least of the order $n^{2k}$, as shown in Lemma \ref{lemma:lower_bound_var_crit_lexi}. An upper bound of the same order can be proved similarly. The variance of the total number of $k$-simplices is also of the same order.
		\item Knowing the expected value and the variance one can prove concentration results using different concentration inequalities, for example, Chebyshev's inequality. This would show that not only the expected value of critical simplices is smaller compared to all simplices but also that large deviations from the mean are unlikely, hence implying that the substantial improvement of one order of $n$ is not only expected but also likely.
		\item For counting critical simplices to high accuracy in probability, it is not necessary to check every simplex. Certain simplices have a very small chance of being critical, and can be safely ignored. The probability of this omission causing an error is vanishingly small asymptotically; see Proposition \ref{proposition:crit_lexi_uptoK}.
	\end{enumerate}
	More details are provided in Remark \ref{remark:TDA}. 
	\subsection*{Related Work}
	
	Theorem \ref{theorem:mvn_dissociated_decomp_approx} is not the first generalisation of the results in \cite{barbour1989central} to a multivariate setting, see for example \cite{fang2016multivariate, raivc2004multivariate}. The key advantage of our approach is that it allows for bounds which are non-uniform in each component of the vector $W$. This is useful when, for example, the number of summands in each component are of different order or when the sizes of dependency neighbourhoods in each component are of different order. The applications consdered here are precisely of this type, where the non-uniformity of the bounds is crucial. Moreover, we do not require the covariance matrix $\Sigma$ to be invertible, and can therefore accommodate degenerate multivariate normal distributions.
	
	Another multivariate central limit theorem for \textit{centered} subgraph counts in the more general setting of a random graph associated to a graphon can be found in \cite{kaur2020higher}. That proof is based on Stein's method via a Stein coupling.   Translating this result for uncentered subgraph counts would yield an approximation by a function of a multivariate normal.  In \cite{reinert2010random}, an exchangeable pair coupling led to \cite[Proposition 2]{reinert2010random} which can be specialised to joint counts of edges and triangles; our approximation  significantly generalises this result beyond the case where $k \in \set{1,2}$.
	Several univariate normal approximation theorems for subgraph counts are available; recent developments in this area include \cite{privault2020normal}, which uses Malliavin calculus together with Stein's method, and  \cite{eichelsbacher2021kolmogorov}, which uses the Stein-Tikhomirov method.
	
	\subsection*{Organisation}
	 In Section \ref{section:main_result} we prove our main approximation theorem using smooth test functions and extend the result to non-smooth test functions using a smoothing technique from \cite{gan2017dirichlet}. In Section \ref{sect:preliminaries}, we recall concepts from the theory of simplicial complexes, which we later use. In Section \ref{section:crit_lexi} we prove an approximation theorem for critical simplex counts of lexicographical matchings. Two technical computations required in this Section have been consigned to the Appendix. In Section \ref{section:link} we prove an approximation theorem for count variables of simplices that are in the link of a fixed simplex. In Section \ref{section:gen_u_stat} we introduce a slight generalisation of generalised $U$-statistics for which Theorem \ref{theorem:mvn_dissociated_decomp_approx} gives a CLT with explicit bounds. We then apply the CLT to simplex counts in the random clique complex.
	 
	 \subsection*{Acknowledgements}
	TT acknowledges funding from EPSRC studentship 2275810. VN is supported by the EPSRC grant EP/R018472/1. GR is funded in part by the EPSRC grants EP/T018445/1 and EP/R018472/1. The authors would like to thank Xiao Fang and Matthew Kahle for helpful discussions.
	
	\section{A Multivariate CLT for Dissociated Sums}\label{section:main_result}

	Throughout this paper we use the following notation. Given positive integers $n,m$ we write $[m,n]$ for the set $\set{m,m+1,\ldots,n}$ and $[n]$ for the set $[1,n]$. Given a set $X$ we write $\abs{X}$ for its cardinality, $\pow{X}$ for its powerset, and given a positive integer $k$ we write $C_k = \setc{t \in \pow{[n]}}{|t| = k}$ for the collection of subsets of $[n]$ which are of size $k$. For a function $f: \RR^d \to \RR$ we write  $\partial_{ij}f = \frac{\partial^2f}{\partial x_i \partial x_j}$ and $\partial_{ijk}f = \frac{\partial^3f}{\partial x_i \partial x_j \partial x_k}$. Also, we write $\abs{f}_k = \sup_{i_1, i_2, \ldots, i_k \in [d]} \linfnorm{\partial_{i_1 i_2 \ldots i_k}f}$ for any integer $k \geq 1$, as long as the quantities exist. Here $|| \cdot ||_\infty$ denotes the supremum norm while $|| \cdot ||_2$ denotes the Euclidean norm. The notation $\nabla$ denotes the gradient operator in $\RR^d$. For a positive integer $d$ we define a class of test functions $h: \RR^d \to \RR$, as follows. We say $h \in \testzero$ iff $h$ is three times partially differentiable with third partial derivatives being Lipschitz and $\abs{h}_3 < \infty$. The notation $\text{\rm Id}_{d \times d}$ denotes the $d \times d$ identity matrix. The vertex set of all graphs and simplicial complexes is assumed to be $[n]$. If $s = (x, i) \in \I$ is an element of the index set in Definition \ref{def:disrand}, then we denote the second component of the tuple by $|s|$, that is $|s| \coloneqq i$. We also use Bachmann-Landau asymptotic notation: we say $f(n) = O(g(n))$ iff $\limsup_{n \to \infty} \frac{|f(n)|}{g(n)} < \infty$ and $f(n) = \Omega(g(n))$ iff $\liminf_{n \to \infty} \frac{f(n)}{g(n)} > 0$.
	
	Throughout this section, $W \in \RR^d$ is a vector of dissociated sums in the sense of Definition \ref{def:disrand}, with covariance matrix whose entries are $\Sigma_{ij} = \Cov(W_i, W_j)$ for $(i,j) \in [d]^2$. For each $s \in \I$ we denote by $\D(s) \subset \I$ the disjoint union $\bigcup_{j=1}^d \D_j(s)$. For each triple $(s,t,j) \in \I^2 \times [d]$ we write the set-difference $\D_j(t) \setminus \D_j(s)$ as $\D_j(t;s)$, with $\D(t;s) \subset \I$ denoting the disjoint union of such differences over $j \in [d]$.
	
	\subsection{Smooth Test Functions}
	To prove Theorem \ref{theorem:mvn_dissociated_decomp_approx} we use Stein's method for multivariate normal distributions; for details see for example Chapter 12 in \cite{chen2011normal}. Our proof of Theorem \ref{theorem:mvn_dissociated_decomp_approx} is based on the {\bf Stein characterization} of the multivariate normal distribution: $Z \in \RR^d$ is a multivariate normal $\mvn{0}{\Sigma}$ if and only if the identity
	\begin{align}\label{eq:steinmnv}
	\EEbracket{\nabla^{T} \Sigma \nabla f(Z)-Z^{T} \nabla f(Z)}=0
	\end{align} holds for all twice continuously differentiable $f: \RR^{d} \to \RR$ for which the expectation exists. 
	In particular, we will use the following result based on \cite[Lemma 1 and Lemma 2]{meckes2009stein}. As Lemma 1 and Lemma 2 in \cite{meckes2009stein} are stated  there only for infinitely differentiable test functions, we give the proof here for completeness.
	
	\begin{lemma}[Lemma 1 and Lemma 2 in \cite{meckes2009stein}]
		\label{lemma:mvn_stein_solution}
		Fix $n \geq 2$. Let $h: \RR^{d} \to \RR$ be $n$ times continuously differentiable with $n$-th partial derivatives being Lipschitz and $Z \sim \mvn{0}{\text{\rm Id}_{d \times d}}$. Then, if $\Sigma \in \RR^{d \times d}$ is symmetric positive semidefinite, there exists a solution $f: \RR^{d} \to \RR$ to the equation
		\begin{equation}\label{equation:stein}
		\nabla^{T} \Sigma \nabla f(w)-w^{T} \nabla f(w)=h(w)-\EE h\left(\Sigma^{1 / 2} Z\right), \quad w \in \RR^{d},
		\end{equation}
		such that $f$ is $n$ times continuously differentiable and we have for every $k=1, \ldots, n$: \[\abs{f}_k \leq \frac{1}{k}\abs{h}_k.\]
		
	\end{lemma}
	
	\begin{proof} 
		Let $h$ be  as in the assertion. It is shown in Lemma 2.1 in \cite{chatterjee2007multivariate}, which is based on a reformulation of Eq.\,(2.20) in  \cite{barbour1990stein}, that  a solution of \eqref{equation:stein} for $h$ is given by $f(x) = f_h(x) = \int_0^1  \frac{1}{2t}  \EE \{ h( Z_{x,t} ) \} \, dt$, with $Z_{x, t}=\sqrt{t} x+\sqrt{1-t} \Sigma^{1 / 2} Z$. As $h$ has $n$-th partial derivatives being Lipschitz and hence for differentiating $f$ we can bring the derivative inside the integral, it is straightforward to see that the solution $f$ is $n$ times continuously differentiable.
		
		The bound on $\abs{f}_k$ is a consequence of 
		\[\frac{\partial^{k} f}{\partial x_{i_{1}} \cdots \partial x_{i_{k}}}(x)=\int_{0}^{1}(2 t)^{-1} t^{k / 2} \EEbracket{\frac{\partial^{k} h}{\partial x_{i_{1}} \cdots \partial x_{i_{k}}}\left(Z_{x, t}\right)} d t\]
		for any $i_1, i_2, \ldots, i_k$; see, for example, Equation (10) in \cite{meckes2009stein}. Taking the sup-norm on both sides and bounding the right hand side of the equation gives 
		\[\norm{\frac{\partial^{k} f}{\partial x_{i_{1}} \cdots \partial x_{i_{k}}}}_{\infty} \leq \norm{\frac{\partial^{k} h}{\partial x_{i_{1}} \cdots \partial x_{i_{k}}}}_{\infty} \int_{0}^{1}(2 t)^{-1} t^{k / 2}dt \leq \frac{1}{k} \abs{h}_k.\]
	\end{proof}

	Note that neither Lemma \ref{lemma:mvn_stein_solution} nor indeed  Theorem \ref{theorem:mvn_dissociated_decomp_approx} require the covariance matrix $\Sigma$ to be invertible.
	
	\begin{proof}[Proof of Theorem \ref{theorem:mvn_dissociated_decomp_approx}]
		To prove Theorem \ref{theorem:mvn_dissociated_decomp_approx}, we replace $w$ by $W$ in Equation \eqref{equation:stein} and take the expected value on both sides. As a result, we aim to bound the expression 
		\begin{align}\label{eq:steinexp}
		\abs{\EEbracket{\nabla^{T} \Sigma \nabla f(W)-W^{T} \nabla f(W)}} = \abs{ \EEbracket{\sum_{i,j=1}^d \partial_{ij}f(W)\Sigma_{ij} - \sum_{i=1}^d W_i \partial_i f(W)}}
		\end{align}
		where $f$ is a solution to the Stein equation \eqref{equation:stein} for the test function $h$. Since the variables $\set{X_s \mid s \in \I}$ are centered and as  $X_t$ is independent of $X_s$ if $t \not\in \D (s)$, for each $(i,j) \in [d]^2$ we have 
		\begin{align}\label{eq:covdecomp}
		\Sigma_{ij} = \Cov(W_i, W_j) = \hspace{-.2em} \sum_{s \in \I_i} \sum_{t \in \D_j(s)} \hspace{-.2em}\EEbracket{X_sX_t}.
		\end{align}
		
		We now use the decomposition of $\Sigma_{ij}$ from \eqref{eq:covdecomp} in the expression \eqref{eq:steinexp}. For each pair $(s,j) \in \I \times [d]$  and $t \in \D(s)$ we set  $ \D_j(t;s)  = \D_j(t) \setminus  \D_j(s) $  and
		\begin{equation}\label{rem:UV} 
		U_j^s \coloneqq \hspace{-.4em} \sum_{u \in \D_j(s)} \hspace{-.4em} X_u;\quad W_j^s \coloneqq W_j - U_j^s, \quad   \mbox{ and } \quad  V_j^{s,t} \coloneqq \hspace{-1em} \sum_{v \in \D_j(t;s)} \hspace{-1em} X_v; \quad W_j^{s,t} \coloneqq W_j^s - V_j^{s,t}. 
		\end{equation} 
		By Definition \ref{def:disrand}, $W_j^s $ is independent of $X_s$, while $W_j^{s,t}$  is independent of the pair $(X_s,X_t)$.
		
		Next we decompose the r.h.s. of \eqref{eq:steinexp};
		\[
		\left|\EEbracket{\sum_{i=1}^d W_i \partial_i f(W) - \sum_{i,j=1}^d \partial_{ij}f(W)\Sigma_{ij}}\right| = \left|R_1 + R_2 + R_3 \right|; 
		\]
		with
		\begin{align}
		R_1 &= \sum_{i=1}^d \EEbracket{W_i \partial_i f(W)} - \sum_{s \in \I} \sum_{j=1}^d  \EEbracket{X_s U^s_j \partial_{|s|j}f(W^s)}, \label{eq:R1}\\
		R_2 &= \sum_{s \in \I} \sum_{j=1}^d \EEbracket{X_s U^s_j \partial_{|s|j}f(W^s)} - \sum_{s \in \I}\sum_{t \in \D_j(s)} \EEbracket{X_sX_t} \EE\partial_{|s||t|}f(W^{s,t}),\label{eq:R2} \text{ and}\\
		R_3 &= \sum_{s \in \I} \sum_{t \in \D(s)}  \EEbracket{X_{s} X_{t}} \left( \EE\partial_{|s||t|}f(W^{s,t}) - \EE\partial_{|s||t|}f(W) \right).
		\end{align}
		Here we recall that if $s=(k,i)$ then  $|s|=i \in [d]$.
		
		As with the vector of dissociated sums $W \in \RR^d$ itself, we can assemble these differences into random vectors. Thus, $W^s \in \RR^d$ is  $(W^s_1,\ldots,W^s_d)$, and similarly $W^{s,t} = (W^{s,t}_1,\ldots W^{s,t}_d)$. In the next three claims, we provide bounds on $R_i$ for $i \in [3]$. 
		
		\begin{claim}\label{lemma:R_1_bound} The absolute value of the expression $R_1$ from \eqref{eq:R1} is bounded above by
			\[
			|R_1| \leq \left(\frac{1}{2} \sum_{s \in \I} \sum_{t \in \D(s)} \sum_{u \in \D(s)} \EE \abs{X_sX_tX_u}\right)  \abs{f}_3.
			\]
		\end{claim}
		\begin{proof}
			Note that
			\begin{align*}
			R_1 &= \sum_{i=1}^d \sum_{s \in \I_i}\EEbracket{X_s \partial_i f(W)} - \sum_{s \in \I} \sum_{j=1}^d  \EEbracket{X_s U^s_j \partial_{|s|j}f(W^s)} & 
			\\
			&= \sum_{i=1}^d \sum_{s \in \I_i}\left(\EEbracket{X_s \partial_i f(W)} - \sum_{j=1}^d  \EEbracket{X_s U^s_j \partial_{ij}f(W^s)}\right). & 
			\end{align*} 
			For each $s \in \I_{i}$, it follows from \eqref{rem:UV} that $W = U^s + W^s$. Using the Lagrange form of the remainder term in Taylor's theorem, we obtain
			\[
			\partial_i f(W) = \sum_{j=1}^d \partial_{ij}f(W^s) U^s_j + \frac{1}{2} \sum_{j,k=1}^d \partial_{ijk}f(W^s + \theta_s U^s) U^s_j U^s_k
			\]
			for some random $\theta_s \in (0,1)$. Using this Taylor expansion in the expression for $R_1$, we get the following four-term summand $S_{i,s}$ for each $i \in [d]$ and $s \in \I_i$:
			\begin{align*}
			S_{i,s} &= \EEbracket{X_s \partial_i f(W^s)}  + \sum_{j=1}^d \EEbracket{X_s\partial_{ij}f(W^s) U^s_j} \\
			& + \frac{1}{2} \sum_{j,k=1}^d \EEbracket{X_s\partial_{ijk}f(W^s + \theta_s U^s) U^s_j U^s_k}
			- \sum_{j=1}^d \EEbracket{X_s  \partial_{ij}f(W^s)U^s_j}.
			\end{align*}
			The second and fourth terms cancel each other. Recalling that $X_s$ is centered by definition and independent of $W^s$ by Definition \ref{def:disrand}, the third term also vanishes and
			\[
			R_1 = \sum_{i=1}^d\sum_{s \in \I_i} S_{i,s} = 
			\frac{1}{2} \sum_{i,j,k=1}^d\sum_{s \in \I_i} \EEbracket{X_s\partial_{ijk}f(W^s + \theta_s U^s) U^s_j U^s_k}.
			\]
			Recalling that $\norm{\partial_{ijk}f}_{\infty} \leq \abs{f}_3$ and that $U^s_j = \sum_{t \in \D_j(s)}X_t$, we have:
			\begin{align*}
			|R_1| &\leq  \frac{1}{2} \sum_{i,j,k=1}^d \sum_{s \in \I_i} \EE \abs{X_s \partial_{ijk}f(W^s + \theta_s U^s) U^s_j U^s_k} \\
			&\leq \frac{\abs{f}_3}{2} \sum_{i,j,k=1}^d \sum_{s \in \I_i} \EE \abs{X_{s} \sum_{t \in \D_j(s)} X_{t} \sum_{u \in \D_k(s)} X_{u}}\\
			&\leq \frac{\abs{f}_3}{2} \sum_{s \in \I} \sum_{t \in \D(s)} \sum_{u \in \D(s)} \EE \abs{X_{s} X_{t} X_{s}},
			\end{align*}
			as desired.
		\end{proof}
		
		\begin{claim}\label{lemma:R_2_bound} The absolute value of the expression $R_2$ from \eqref{eq:R2} is bounded above by
			\[
			|R_2| \leq \left( \sum_{s \in \I} \sum_{t \in \D(s)} \sum_{u \in \D(t;s)}  \EE \abs{X_{s} X_{u} X_{r}} \right)  \abs{f}_3.
			\]
		\end{claim}
		\begin{proof}
			
			Recalling that $U^s_j = \sum_{t \in \D_j(s)} X_{t}$ and $\D(s) = \bigcup_{j=1}^d \D_j(s)$,
			\begin{align*}
			R_2 &= \sum_{s \in \I} \sum_{t \in \D(s)} \left\{ \EEbracket{X_{s} X_{t} \partial_{\abs{s}\abs{t}}f(W^s)} - \EEbracket{X_{s}X_{t}} \EEbracket{\partial_{\abs{s}\abs{t}}f(W^{s,t})} \right\}.
			\end{align*} 
			
			Fix $s \in \I$ and $t \in \D_j(s)$.
			Recall that by \eqref{rem:UV}, $W^s = W^{s,t} + V^{s,t}$. Using the Lagrange form of the remainder term in Taylor’s theorem, we obtain: 
			\[
			\partial_{\abs{s}\abs{t}} f(W^s) = \partial_{\abs{s}\abs{t}} f(W^{s,t}) + \sum_{k=1}^d \partial_{\abs{s}\abs{t}k}f(W^{s,t} + \theta_{s,t} V^{s,t}) V^{s,t}_k
			\]
			for some random $\theta_{s,t} \in (0,1)$. Using this Taylor expansion in the expression for $R_2$, we get the following three-term summand $S_{s,t}$ for each pair $(s, t) \in \I \times \D_j(s)$:
			\begin{align*}
			S_{s,t} &= \EEbracket{X_{s} X_{t} \partial_{\abs{s}\abs{t}}f(W^{s,t})} + \sum_{k=1}^d \EEbracket{X_{s} X_{t} \partial_{\abs{s}\abs{t}k}f(W^{s,t} + \theta_{s,t} V^{s,t}) V^{s,t}_k}\\
			&-\EEbracket{X_{s}X_{t}} \EEbracket{\partial_{\abs{s}\abs{t}}f(W^{s,t})}.
			\end{align*}
			
			Recalling that $W^{s,t}$ is independent of the pair $(X_s, X_t)$
			the first and the last terms cancel each other and only the sum over $k$ is left:
			\begin{align*}
			R_2 &= \sum_{s \in \I} \sum_{t \in \D(s)} S_{s,t} = \sum_{s \in \I} \sum_{t \in \D(s)} \sum_{k=1}^d \EEbracket{X_{s} X_{t} \partial_{\abs{s}\abs{t}k}f(W^{s,t} + \theta_{s,t} V^{s,t}) V^{s,t}_k}.
			\end{align*}
			
			Recalling that $\norm{\partial_{ijk}f}_{\infty} \leq \abs{f}_3$ and that $V_k^{s,t} = \sum_{v \in \D_k(t;s)}  X_v$ we have:
			\begin{align*}
			|R_2| 
			&\leq \sum_{s \in \I} \sum_{t \in \D(s)} \sum_{k=1}^d \sum_{v \in \D_k(t;s)} \EE \abs{X_{s} X_{t} X_{v} \partial_{\abs{s}\abs{t}k}f(W^{s,t} + \theta_{s,t} V^{s,t})} \\
			& \leq \abs{f}_3 \sum_{s \in \I} \sum_{t \in \D(s)} \sum_{u \in \D(t;s)}  \EE \abs{X_{s} X_{u} X_{r}},
			\end{align*}
			as required.
			
		\end{proof}
		
		\begin{claim}\label{lemma:R_3_bound}
			\begin{align*}
			|R_3| \leq \left( \sum_{s \in \I} \sum_{t \in \D(s)} \left\{ \sum_{u \in \D(s)} \EE \abs{X_{s} X_{t}} \EE \abs{X_{u}} +  \sum_{u \in \D(t;s)} \EE \abs{X_{s} X_{t}} \EE \abs{X_{u}} \right\}\right)  \abs{f}_3 .
			\end{align*}
		\end{claim}
		\begin{proof}
			
			Fix $(s, t) \in \I \times \D_j(s)$. Recall that by \eqref{rem:UV}, $W^{s,t} = W - U^s - V^{s,t}$. Using the Lagrange form of the remainder term in Taylor’s theorem, we obtain
			\[
			\partial_{\abs{s}\abs{t}}f(W^{s,t}) = \partial_{\abs{s}\abs{t}}f(W) - \sum_{k=1}^d \partial_{\abs{s}\abs{t}k}f(W - \rho_{s,t}(U^s + V^{s,t}))(U^s_k + V^{s,t}_k)
			\]
			for some random $\rho_{s,t} \in (0,1)$. Recalling that $U_k^s = \sum_{t \in \D_k(s)} X_t$ and $V_k^{s,t} =\sum_{u \in \D_j(t;s)}  X_u$,
			\begin{align*}
			R_3 
			= &-\sum_{s \in \I} \sum_{t \in \D(s)} \sum_{k=1}^d \EEbracket{X_{s} X_{t}} \EEbracket{\partial_{\abs{s}\abs{t}k}f(W - \rho_{s,t}(U^s + V^{s,t}))(U^s_k + V^{s,t}_k) }\\
			=& - \sum_{s \in \I} \sum_{t \in \D(s)} \sum_{u \in \D(s)} \EEbracket{X_{s} X_{t}} \EEbracket{X_{u} \partial_{|s||t|k}f(W - \rho_{s,t}(U^s + V^{s,t}))} \\
			&-\sum_{s \in \I} \sum_{t \in \D(s)} \sum_{u \in \D(t;s)} \EEbracket{X_{s} X_{t}} \EEbracket{X_{u} \partial_{|s||t|k}f(W - \rho_{s,t}(U^s + V^{s,t}))}.
			\end{align*}

			Recalling that $\norm{\partial_{ijk}f}_{\infty} \leq \abs{f}_3$ we bound:
			\begin{align*}
			|R_3| 
			\leq & \abs{f}_3 \sum_{s \in \I} \sum_{t \in \D(s)} \sum_{u \in \D(s)} \EE \abs{X_{s} X_{t}} \EE \abs{X_{u}} + \abs{f}_3 \sum_{s \in \I} \sum_{t \in \D(s)} \sum_{u \in \D(t;s)} \EE \abs{X_{s} X_{t}} \EE \abs{X_{u}},
			\end{align*}
			as required.
		\end{proof}
		
		Take any $h \in \testzero$. Let $f: \RR^d \to \RR$ be the associated solution from Lemma \ref{lemma:mvn_stein_solution}. Combining Claims \ref{lemma:mvn_stein_solution} - \ref{lemma:R_3_bound} and using Lemma \ref{lemma:mvn_stein_solution} we have:
		\begin{align*}
		&\abs{\EE h(W) - \EE h(\Sigma^{\frac{1}{2}}Z)} \\
		\leq&  \abs{\EEbracket{\nabla^{T} \Sigma \nabla f(W)-W^{T} \nabla f(W)}} \leq |R_1| + |R_2| + |R_3| \\
		\leq  & \abs{f}_3 \sum_{s \in \I} \sum_{t \in \D(s)} \sum_{u \in \D(s)}  \left(\frac{1}{2} \EE \abs{X_{s} X_{t} X_{u}} + \EE \abs{X_{s} X_{t}} \EE \abs{X_{u}} \right) \\
		&+\abs{f}_3 \sum_{s \in \I} \sum_{t \in \D(s)} \sum_{u \in \D(t;s)} \left( \EE \abs{X_{s} X_{t} X_{u}} + \EE \abs{X_{s} X_{t}} \EE \abs{X_{u}} \right) \\
		\leq & \frac{1}{3} \abs{h}_3 B_{\ref{theorem:mvn_dissociated_decomp_approx}}.
		\end{align*}
		
	\end{proof}

	\medskip 
	In most of our applications, the variables $X_s$ are centered and rescaled Bernoulli random variables. Hence, the following lemma is useful.
	
	\begin{lemma}\label{lemma:unified_moments}
		Let $\xi_1, \xi_2, \xi_3$ be Bernoulli random variables with expected values $\mu_1, \mu_2, \mu_3$ respectively. Let $c_1, c_2, c_3 > 0$ be any constants. Consider variables $X_i \coloneqq c_i (\xi_i - \mu_i)$ for $i = 1,2,3$. Then we have
		\begin{align*}
		&\EE \abs{X_1 X_2 X_3} \leq c_1 c_2 c_3 \left\{ \mu_1\mu_2(1-\mu_1)(1-\mu_2) \right\}^{\frac{1}{2}}; \\
		&\EE \abs{X_{1} X_{2}} \EE \abs{X_{3}} \leq c_1 c_2 c_3 \left\{ \mu_1\mu_2(1-\mu_1)(1-\mu_2) \right\}^{\frac{1}{2}}.
		\end{align*}
	\end{lemma}
	
	\begin{proof}
		Note that $X_{3}$ can take two values: $-c_3 \mu_3$ or $c_3(1-\mu_3)$. As $0 \leq \mu_3 \leq 1$, we have
		\[\EE \abs{X_{1} X_{2}} \EE \abs{X_{3}} \leq c_3 \EE \abs{X_{1} X_{2}};\] \[\EE \abs{X_1 X_2 X_3} \leq c_3 \EE \abs{X_1 X_2}.\] Applying the Cauchy-Schwarz inequality and direct calculation of the second moments gives
		\[\EE \abs{X_1 X_2} \leq \left\{\EEbracket{X_1^2} \EEbracket{X_2^2}\right\}^{\frac{1}{2}} = c_1c_2 \left\{ \mu_1\mu_2(1-\mu_1)(1-\mu_2) \right\}^{\frac{1}{2}},\] which finishes the proof.
	\end{proof}

	\subsection{Non-smooth Test Functions}
	Here we follow \cite[Section 5.3]{kaur2020higher} very closely to derive a bound on the convex set distance between a vector of dissociated sums $W \in \RR^d$ with covariance matrix $\Sigma$ and a target multivariate normal distribution $\Sigma^{\frac{1}{2}}Z$, where $Z \sim \mvn{0}{\text{\rm Id}_{d \times d}}$. The smoothing technique used here is introduced in \cite{gan2017dirichlet}. However, a better  (polylogarithmic) dependence on $d$ could potentially be achieved using a recent result \cite[Proposition 2.6]{gaunt2022bounding}, at the expense of larger constants.  The recursive approach from \cite{schulte2019multivariate,kasprzak2022vector}  usually yields better dependence on $n$; however, this requires the target normal distribution to have an invertible covariance matrix. Since this property does not always hold in our applications of interest, we do not	use the recursive approach here. To state our next result, let$\mathcal{K}$ be a class of convex sets in $\RR^d$. 
	
	\begin{theorem}\label{theorem:mvn_dissociated_decomp_approx_convex_sets}
		Consider a standard $d$-dimensional Gaussian vector $Z \sim \mvn{0}{\text{\rm Id}_{d \times d}}$. For any centered vector of dissociated sums $W \in \RR^d$ with a positive semi-definite covariance matrix $\Sigma$ and finite third absolute moments we have 
		\[
		\sup _{A \in \mathcal{K}}|\PP(W \in A)-\PP(\Sigma^{\frac{1}{2}}Z \in A)| \leq 2^{\frac{7}{2}} 3^{-\frac{3}{4}}d^{\frac{3}{16}}B_{\ref{theorem:mvn_dissociated_decomp_approx}}^{\frac{1}{4}},
		\]
		where the quantity $B_{\ref{theorem:mvn_dissociated_decomp_approx}}$ as in Theorem \ref{theorem:mvn_dissociated_decomp_approx}.
	\end{theorem}
	
	\begin{proof}
		Fix $A \in \mathcal{K}$, $\epsilon > 0$ and define 
		\[
		A^{\epsilon}=\left\{y \in \RR^{d}: d(y, A)<\epsilon\right\}, \quad \text { and } \quad A^{-\epsilon}=\left\{y \in \RR^{d}: B(y ; \epsilon) \subseteq A\right\}
		\]
		where $d(y, A)=\inf _{x \in A}\norm{ x-y }_2$ and $B(y ; \epsilon) = \setc{z \in \RR^d}{\norm{y-z}_2 \leq \epsilon}$. 
		
		Let $ \mathcal{H}_{\epsilon, A} \coloneqq \set{h_{\epsilon, A}: \RR^{d} \to [0,1] ; A \in \mathcal{K}}$ be a class of functions such that $h_{\epsilon, A}(x)=1$ for $x \in A$ and 0 for $x \notin A^{\epsilon}$. Then, by \cite[Lemma 2.1]{bentkus2003dependence} as well as inequalities (1.2) and (1.4) from \cite{bentkus2003dependence}, for any $\epsilon > 0$ we have:
		\[
		\sup _{A \in \mathcal{K}}|\PP(W \in A)-\PP(\Sigma^{\frac{1}{2}}
		Z \in A)| \leq 4 d^{\frac{1}{4}} \epsilon+\sup _{A \in \mathcal{K}}\left|\EE h_{\epsilon, A}(W)-\EE h_{\epsilon, A}(\Sigma^{\frac{1}{2}}Z)\right|
		\]
		
		Let $f: \RR^{d} \to \RR$ be a 
		bounded Lebesgue measurable function, and for $\delta>0$  let
		\[
		\left(S_{\delta} f\right)(x)=\frac{1}{(2 \delta)^{d}} \int_{x_{1}-\delta}^{x_{1}+\delta} \cdots \int_{x_{d}-\delta}^{x_{d}+\delta} f(z) d z_{d} \ldots d z_{1}. 
		\]
		Set $\delta=\frac{\epsilon}{16\sqrt{d}}$ and $h_{\epsilon, A}=S_{\delta}^{4} I_{A^{\epsilon/4}}$, where $I_{A^{\epsilon/4}}$ is the indicator function of the subset $A^{\epsilon/4} \subseteq \RR^d$. By \cite[Lemma 3.9]{gan2017dirichlet} we have that $h_{\epsilon, A}$ is bounded and has three continuous bounded partial derivatives and its third partials are Lipschitz. Moreover, the following bounds hold:
		\[
		\left\|h_{\epsilon, A}\right\|_{\infty} \leq 1, \quad\left|h_{\epsilon, A}\right|_{2} \leq \frac{1}{\epsilon^{2}},  \quad\left|h_{\epsilon, A}\right|_{3} \leq \frac{1}{\epsilon^{3}}.
		\]
		
		Note that $h_{\epsilon, A}=S_{\delta}^{4} I_{A^{\epsilon/4}} \in \mathcal{H}_{\epsilon, A}$ and hence \cite[Lemma 2.1]{bentkus2003dependence} applies. Using this with Theorem \ref{theorem:mvn_dissociated_decomp_approx} we get:
		\begin{align*}
		\sup _{A \in \mathcal{K}}&|\PP(W \in A)-\PP(\Sigma^{\frac{1}{2}}Z \in A)| \\
		\leq & 4 d^{\frac{1}{4}} \epsilon+\sup _{A \in \mathcal{K}}\left|\EE h_{\epsilon, A}(W)-\EE h_{\epsilon, A}(\Sigma^{\frac{1}{2}}Z)\right|\\
		\leq & 4 d^{\frac{1}{4}} \epsilon + \frac{1}{3\epsilon^3} B_{\ref{theorem:mvn_dissociated_decomp_approx}}.
		\end{align*}
		
		Since this bound works for every $\epsilon > 0$, we minimise it by using $\epsilon = \left(\frac{3B_{\ref{theorem:mvn_dissociated_decomp_approx}}}{4d^{\frac{1}{4}}}\right)^{\frac{1}{4}}$.
		
	\end{proof}
	
	The next result provides a simplification of Theorems \ref{theorem:mvn_dissociated_decomp_approx} and \ref{theorem:mvn_dissociated_decomp_approx_convex_sets} under the assumption that one uses bounds that are uniform in $s, t, u \in \I$. Its proof follows immediately from writing the sum over $ \sum_{s \in \I} \sum_{t,u \in \D(s)} $ as the sum over $\sum_{i \in [d]} \sum_{j \in [d]} \sum_{k \in [d]} \sum_{s \in \I_i} \sum_{t \in \D_j(s)} \sum_{u \in \D_k(s)}$.
	
	\begin{corollary}\label{corollary:mvn_dissociated_decomp_approx}
		We have the following two bounds:
		\begin{enumerate}
			\item Under the assumptions of Theorem \ref{theorem:mvn_dissociated_decomp_approx},
			\[
			\abs{\EE h(W) - \EE h(\Sigma^{\frac{1}{2}}Z)} \leq B_{\ref{corollary:mvn_dissociated_decomp_approx}}  \abs{h}_3.
			\]
			\item Assuming the hypotheses of Theorem \ref{theorem:mvn_dissociated_decomp_approx_convex_sets},
			\[
			\sup _{A \in \mathcal{K}}|\PP(W \in A)-\PP(\Sigma^{\frac{1}{2}}Z \in A)| \leq 2^{\frac{7}{2}} 3^{-\frac{3}{4}}d^{\frac{3}{16}}B_{\ref{corollary:mvn_dissociated_decomp_approx}}^{\frac{1}{4}}.
			\]
		\end{enumerate}
		Here $B_{\ref{corollary:mvn_dissociated_decomp_approx}}$ is a sum over $(i,j,k) \in [d]^3$ of the form
		\[
		B_{\ref{corollary:mvn_dissociated_decomp_approx}} \coloneqq \frac{1}{3} \sum_{(i,j,k)} \abs{\I_i}  \alpha_{ij}  \left(\frac{3\alpha_{ik}}{2} + 2\alpha_{jk}\right)  \beta_{ijk};
		\]
		and $\alpha_{ij}$ is the largest  value attained by $\abs{\D_j(s)}$ over $s \in \I_i$, and 
		\[
		\beta_{ijk} = \max_{s,t,u}\Big(\EE\abs{X_sX_tX_u},\EE\abs{X_sX_t}\EE\abs{X_u}\Big)
		\] as $(s,t,u)$ range over $\I_i \times \I_j \times \I_k$.
	\end{corollary}
	
	\section{Simplicial Complex Preliminaries}\label{sect:preliminaries}
	
	\subsection{First definitions}
	
	Firstly, we recall the notion of a simplicial complex \cite[Ch 3.1]{spanier}; these provide  higher-dimensional generalisations of a graph and constitute data structures of interest across algebraic topology in general as well as applied and computational topology in particular. 
	
	A simplicial complex $\cplx$ on a vertex set $V$ is a set of nonempty subsets of $V$ (i.e. $\varnothing \notin \cplx \subseteq \pow{V}$) such that the following properties are satisfied:
	\begin{enumerate}
		\item for each $v \in V$ the singleton $\set{v}$ lies in $\cplx$, and
		\item if $t \in \cplx$ and $s \subset t$ then $s \in \cplx$.
	\end{enumerate}
	
	The \textbf{dimension} of a simplicial complex $\cplx$ is $\max_{s \in \cplx}|s| - 1$. Elements of a simplicial complex are called \textbf{simplices}. If $s$ is a simplex, then its dimension is $|s| - 1$. A simplex of dimension $k$ can be called a $k$-simplex. Note that the notion of one-dimensional simplicial complex is equivalent to the notion of a graph, with the vertex set $V$ and edges as subsets. 
	
	Given a graph $G = (V, E)$ the clique complex $\mathcal{X}$ of $G$ is a simplicial complex on $V$ such that \[ t \in \mathcal{X}  \iff \forall u,v \in t,  \; \set{u,v} \in E
	.\]
	Recall that $\gnp$ is a random graph on $n$ vertices where each pair of vertices is connected with probability $p$, independently of any other pair. The $\xnp$ random simplicial complex is the clique complex of the $\gnp$ random graph, which is a random model studied in stochastic topology \cite{kahle2009topology, kahle2014sharp}. Note that $t \in {\mathcal{X}} $ if and only if the vertices of $t$ span a clique in $G$. Thus, elements in $\xnp$ are cliques in $\gnp$.
	
	\subsection{Links}
	
	The link of a simplex $t$ in a simplicial complex $\cplx$ is the subcomplex \[
	\lk(t) = \setc{s \in \cplx}{s \cup t \in \cplx \text{ and } t \cap s = \varnothing}.
	\]
	
	\begin{figure}
		\centering
		\includegraphics[width=0.7\textwidth]{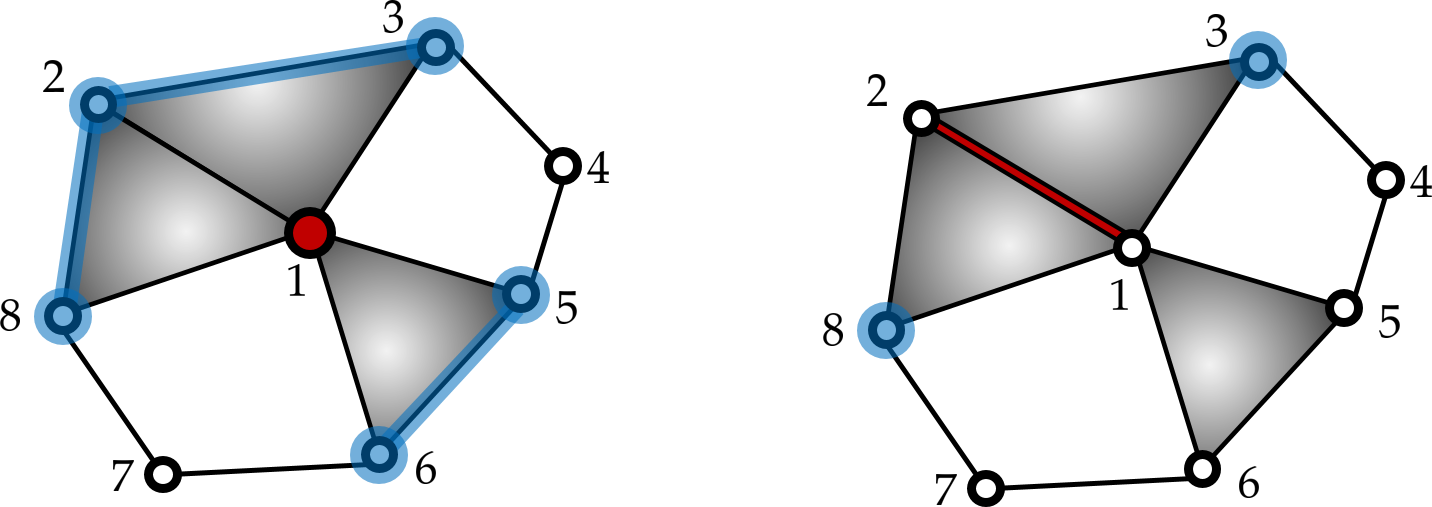}
		\caption{Left: the link (highlighted in blue) of the vertex 1 (highlighted in red). Right: the link (highlighted in blue) of the edge $\set{1,2}$ (highlighted in red). The two-dimensional simplices are shaded in grey.}
		\label{fig:links}
	\end{figure}
	
	\begin{example}
		If we look at a graph as a one dimensional simplicial complex, then the vertices are sets of the form $\set{i}$ and edges are sets of the form $\set{i,j}$. For a vertex $t = \set{v}$, the edges of the form $s = \set{v,u}$ will not be in the link of $t$ because $t \cap s = \varnothing$ is not satisf\textbf{}ied. If we pick $s = \set{i,j}$ and $v \notin s$, then $s \cup t \in \cplx$ is not satisfied. So there will be no edges in the link. However, if $s = \set{u}$ and $u$ is a neighbour of $v$, then $s \cup t \in \cplx$ and $s \cap t = \varnothing$. Hence the link of a vertex will be precisely the other vertices that the vertex is connected to; the notion of the link generalises the idea of a neighbourhood in a graph.
	\end{example}
	
	\begin{example} 
		Now consider the simplicial complex depicted in Figure \ref{fig:links}: it has 8 vertices, 12 edges and 3 two-dimensional simplices that are shaded in grey. On the left hand side of the figure we see highlighted in blue the link of the vertex 1, which is highlighted in red. So $\lk(\set{1}) = \set{\set{2}, \set{3}, \set{5}, \set{6}, \set{8}, \set{2,3}, \set{2,8}, \set{5,6}}$. On the right hand side of the figure we see highlighted in blue the link of the edge $\set{1,2}$, which is highlighted in red. That is, $\lk(\set{1,2}) = \set{\set{3}, \set{8}}$.
	\end{example}
	
	\subsection{Discrete Morse theory}

	A \textbf{partial matching} on a simplicial complex $\cplx$ is a collection
	\[
	\Sigma = \setc{(s, t)}{s \subseteq t \in \cplx \text{ and } |t| - |s| = 1}
	\] such that every simplex appears in at most one pair of $\Sigma$. A \textbf{$\mathbf{\Sigma}$-path} (of length $k \geq 1)$ is a sequence of distinct simplices of $\cplx$ of the following form:
	\[(s_1 \subseteq t_1 \supseteq s_2 \subseteq t_2 \supseteq \ldots \supseteq s_k \subseteq t_k)\]
	such that $(s_i, t_i) \in \Sigma$ and $|t_i| - |s_{i+1}| = 1$ for all $i \in [k]$. A $\Sigma$-path is called a \textbf{gradient path} if $k=1$ or $s_1$ is not a subset of $t_k$. A partial matching $\Sigma$ on $\cplx$ is called \textbf{acyclic} iff every $\Sigma$-path is a gradient path. Given a partial matching $\Sigma$ on $\cplx$, we say that a simplex $t \in \cplx$ is \textbf{critical} iff $t$ does not appear in any pair of $\Sigma$.

	For a one-dimensional simplicial complex, viewed as a graph, a partial matching $\Sigma$ is comprised of elements $( v; \{u,v\})$ with $v$ a vertex and $\{u, v\}$ an edge. A $\Sigma-$path is then a sequence of distinct vertices and edges 
	\[
	v_1, \set{v_1,v_2}, v_2, \set{v_2,v_3}, \ldots, v_k, \set{v_k,v_{k+1}}
	\]
	where each consecutive pair of the form $(v_i,\set{v_i,v_{i+1}})$ is constrained to lie in $\Sigma$.
	
	\medskip
	We refer the interested reader to \cite{forman2002user} for an introduction to discrete Morse theory and to \cite{mischaikow2013morse} for seeing how it is used to reduce computations in the persistent homology algorithm. In this work we aim to understand how much improvement one would likely get on a random input when using a specific type of acyclic partial matching, defined below. 
	
	\begin{definition}\label{def:lexi_matching}
		Let $\cplx$ be  a simplicial complex and assume that the vertices are ordered by $[n] = \{1,\ldots,n\}$.
		For each simplex $s \in \cplx$ define
		\[
		I_{\cplx}(s) \coloneqq \{j \in [n] \mid j < \min(s) \text{ and } s \cup \{j\} \in \cplx\}. 
		\]
		Now consider the pairings
		\[
		s \leftrightarrow s \cup \{i\},
		\]
		where $i = \min I_{\cplx}(s)$ is the smallest element in the set $ I_{\cplx}(s)$, defined whenever $I_{\cplx}(s) \neq \varnothing$. We call this the {\bf lexicographical matching}.
	\end{definition}
	
	Due to the  $\min I_{\cplx}(s)$ construction in the lexicographical matching, the indices are decreasing along any path and hence it will be a gradient path, showing that the lexicographical matching is indeed an acyclic partial matching on $\cplx$.

	\begin{figure}
		\centering
		\includegraphics[width=0.25\textwidth]{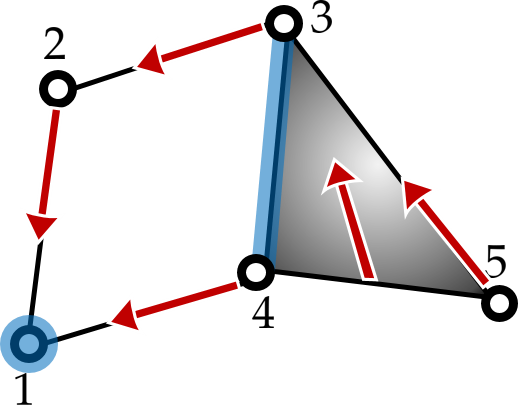}
		\caption{Lexicographical matching given by the red arrows. Critical simplices are highlighted in blue.}
		\label{fig:lexmorse}
	\end{figure}
	
	\begin{example}
		Consider the simplicial complex $\cplx$ depicted in Figure \ref{fig:lexmorse}. 
		The complex has 5 vertices, 6 edges and one two-dimensional simplex that is shaded in grey. The red arrows show the lexicographical matching on this simplicial complex: there is an arrow from a simplex $s$ to $t$ iff the pair $(s, t)$ is part of the matching. More explicitly, the lexicographical matching on $\cplx$ is \[\Sigma = \set{(\set{2}, \set{1,2}), (\set{3}, \set{2,3}), (\set{4}, \set{1,4}), (\set{5}, \set{3,5}), (\set{4, 5}, \set{3,4, 5})}.\] 
		Note that $\set{3,4}$ cannot be matched because the set $I_{\cplx}(\set{3,4})$ is empty. Also, in any lexicographical matching $\set{1}$ is always critical as there are no vertices with a smaller label and hence the set $I_{\cplx}(\set{1})$ is empty. So under this matching there are two critical simplices: $\set{1}$ and $\set{3,4}$, highlighted in blue in the figure. Hence, if we were computing the homology of this complex, considering only two simplices would be sufficient instead of all 12 which are in $\cplx$ - a significant improvement.
	\end{example}
	
	\section{Critical Simplex Counts for Lexicographical Morse Matchings}\label{section:crit_lexi}
	Now we attend to our motivating problem, critical simplex counts. Consider the random simplicial complex $\xnp$. In this section we study the joint distribution of critical simplices in different dimensions with respect to the lexicographical matching on $\xnp$. We start with the following lemma, which is an immediate consequence of Definition \ref{def:lexi_matching}, allowing us to write down the variables of interest in terms of the edge indicators.
	
	\begin{lemma}
		\label{proposition:critical_lexi}
		Let $\cplx$ be a simplicial complex. Consider the lexicographical matching on $\cplx$. Then $t \in \cplx$ matches with one of its cofaces (i.e. $s \in \cplx$ with $|s| - |t| = 1$ and $t \subset s$) iff it is not the case that for all $j < \min(t)$ we have $t \cup \set{j} \notin \cplx$. Also, $t \in \cplx$ matches with one of its faces (i.e. $s \in \cplx$ with $|t| - |s| = 1$ and $s \subset t$) iff for all $j < \min(t)$ we have $t \setminus \set{\min(t)} \cup \set{j} \notin \cplx$. 
	\end{lemma}
	
  For any pair of integers $1 \leq i < j \leq n$ let $Y_{i,j} \coloneqq \ind{\set{i,j} \in \xnp}$ be the edge indicator. Fix $s \in C_k$. Define the variables $X_s^{+} = \ind{s \text{ matches with its coface given it is a simplex}}$ and $X_s^{-} = \ind{s \text{ matches with its face given it is a simplex}}$. The events that the two variables indicate are disjoint. By Lemma \ref{proposition:critical_lexi} we can see that $X_s^{+} = 1 - \prod_{i=1}^{\min(s) - 1}\left(  1 - \prod_{j \in s} Y_{i,j} \right)$ and $X_s^{-} = \prod_{i=1}^{\min(s) - 1}\left(  1 - \prod_{j \in s_{-}} Y_{i,j}\right)$, where $s_{-} \coloneqq s \setminus \set{\min(s)}$. Hence, 
	\begin{align*}
	&\ind{s \text{ is a critical simplex}} = \ind{s \in \xnp}(1 - (X_s^{+} + X_s^{-})) \\
	&= \prod_{i \neq j \in s} Y_{i,j} \left[ \prod_{i=1}^{\min(s) - 1}\left(1 - \prod_{j \in s}Y_{i,j}\right) - \prod_{i=1}^{\min(s) - 1}\left(  1 - \prod_{j \in s_{-}} Y_{i,j}\right)  \right].
	\end{align*}
	
	Thus, the random variable of interest, counting the number of $(k-1)$-simplices that are critical under the lexicographical matching, is
	\begin{equation}\label{equation:def_lexi_crit}
	T_k = \sum_{s \in C_k} \prod_{i \neq j \in s} Y_{i,j} \left[ \prod_{i=1}^{\min(s) - 1}\left(1 - \prod_{j \in s}Y_{i,j}\right) - \prod_{i=1}^{\min(s) - 1}\left(  1 - \prod_{j \in s_{-}} Y_{i,j}\right)  \right].
	\end{equation}
	Note that this random variable does not fit into the framework of generalised $U$-statistics, which we will discuss in Section \ref{section:gen_u_stat}, because the summands in $T_k$ depend not only on the variables that are indexed by the subset $s$.

	\subsection{Moments}
	
	\begin{lemma}
		\label{lemma:expect_crit_lexi}
		For any $1 \leq k \leq n - 1$ we have:
		
		\[ p^{\binom{k+1}{2}+k}\binom{n-2}{k}(1-p) \leq \EE\{ T_{k+1} \} \leq p^{\binom{k+1}{2}-k-1}\binom{n-1}{k}(1-p).\]
	\end{lemma}
	
	\begin{proof}
		\begin{align*}
		\EE\ \{T_{k+1}\} & = \sum_{l=1}^{n-k} \sum_{\substack{s \in C_{k+1} \\ \min(s) = l}} \EEbracket{\prod_{i \neq j \in s} Y_{i,j} \left[ \prod_{i=1}^{l - 1}\left(1 - \prod_{j \in s}Y_{i,j}\right) - \prod_{i=1}^{l - 1}\left(  1 - \prod_{j \in s_{-}} Y_{i,j}\right)  \right]} \\
		&= p^{\binom{k+1}{2}} \sum_{l=1}^{n-k} \sum_{\substack{s \in C_{k+1} \\ \min(s) = l}} \left\{  (1-p^{k+1})^{l - 1} - (1-p^k)^{l-1} \right\} \\
		&=p^{\binom{k+1}{2}} \sum_{l=0}^{n-k-1} \binom{n-l-1}{k} \left\{  (1-p^{k+1})^{l} - (1-p^k)^{l} \right\} \\
		& \leq p^{\binom{k+1}{2}} \binom{n-1}{k}  \sum_{l=0}^\infty \left\{  (1-p^{k+1})^{l} - (1-p^k)^{l} \right\}\\
		& = p^{\binom{k+1}{2}-k-1}\binom{n-1}{k}(1-p).
		\end{align*}
		
		Moreover, 
		\begin{align*}\EE\ \{T_{k+1}\} &= p^{\binom{k+1}{2}} \sum_{l=0}^{n-k-1} \binom{n-l-1}{k} \left\{  (1-p^{k+1})^{l} - (1-p^k)^{l} \right\} \\
		& \geq p^{\binom{k+1}{2}} \binom{n-2}{k} \left\{ (1-p^{k+1})^1 - (1-p^k)^1 \right\}\\
		& = p^{\binom{k+1}{2}+k}\binom{n-2}{k}(1-p).
		\end{align*}
	\end{proof}
	
	In this example, bounding the variance is not immediate. The proof of the following Lemmas \ref{lemma:var_crit_lexi} and \ref{lemma:lower_bound_var_crit_lexi} are long (and not particularly insightful) calculations, which are deferred to the Appendix.
	
	\begin{lemma}\label{lemma:var_crit_lexi}
		For any integer $1 \leq k \leq n - 1$ we have: \[\Var\{ T_{k+1} \} = 2p^{2\binom{k+1}{2}}V_1 + 2p^{2\binom{k+1}{2}}V_2 + p^{2\binom{k+1}{2}}V_3 +  p^{\binom{k+1}{2}}V_4,\]
		
		where
		
		\begin{align*}
		V_1 = &\sum_{i < j}^{n-k} \sum_{m=1}^k \sum_{q=1}^{\min(k+1, j-i)} \binom{n-j}{2k+1-m-q} \binom{2k+1-m-q}{k} \binom{k}{m-1} \binom{j-i+1}{q-1} \\
		&\Big\{ \theta(i,j,q,m,1) \big[(1-2p^{k+1} + p^{2k+2-m})^{i-1} - (1-p^{k+1}-p^k + p^{2k+1-m})^{i-1}\big]\\
		&+ \theta(i,j,q,m,0) \big[(1-2p^{k} + p^{2k+1-m})^{i-1} - (1-p^{k+1}-p^k + p^{2k+2-m})^{i-1}\big] - \eta(i)\eta(j)\Big\} ;\\
		V_2 = &\sum_{i < j}^{n-k} \sum_{m=1}^k \sum_{q=1}^{\min(k+1, j-i)} \binom{n-j}{2k+1-m-q} \binom{2k+1-m-q}{k} \binom{k}{m} \binom{j-i+1}{q-1} \\
		&\Big\{ \theta(i,j,q,m,1) \big[(1-2p^{k+1} + p^{2k+2-m})^{i-1} - (1-p^{k+1}-p^k + p^{2k+2-m})^{i-1}\big]\\
		&+ \theta(i,j,q,m,0) \big[(1-2p^{k} + p^{2k-m})^{i-1}  - (1-p^{k+1}-p^k + p^{2k+2-m})^{i-1}\big] - \eta(i)\eta(j)\Big\} ;\\
		V_3 = &\sum_{i=1}^{n-k} \sum_{m=1}^k  \binom{n-i}{2k+1-m} \binom{2k+1-m}{k} \binom{k}{m-1} \\
		&\Big\{ p^{- \binom{m}{2}} \Big[(1-2p^{k+1} + p^{2k+2-m})^{i-1} +\\
		& (1-2p^{k} + p^{2k+1-m})^{i-1} - 2(1 -p^k -p^{k+1} + p^{2k+2-m})^{i-1} \Big] -  \eta(i)^{2}\Big\}; \\
		V_4 = &\sum_{i=1}^{n-k} \binom{n-i}{k} \left\{ \eta(i) - p^{\binom{k+1}{2}} \eta(i)^{2}\right\}.
		\end{align*}
		
		Here we have used the following notation:
		\begin{align*}
		\eta(a) &\coloneqq (1-p^{k+1})^{a-1} - (1-p^k)^{a-1};\\
		\theta(i,j,q,m,\delta) &\coloneqq p^{- \binom{m}{2}}(1-p^{k+\delta})^{j-i-q}(1-p^{k+\delta-m})^q.
		\end{align*}
		Also, $ \sum_{i < j}^{n-k}$ stands for $ \sum_{i=1}^{n-k-1} \sum_{j= i+1}^{n-k}$.
	\end{lemma}
	
	\begin{lemma}\label{lemma:lower_bound_var_crit_lexi}
		For a fixed integer $1 \leq k \leq n-1$ and $p \in (0,1)$ there is a constant $C_{p,k} > 0$ independent of $n$ and a natural number $N_{p,k}$ such that for any $n \geq N_{p,k}$:
		\[\Var(T_{k+1}) \geq C_{p,k}n^{2k}.\]
	\end{lemma}
	In Lemma \ref{lemma:lower_bound_var_crit_lexi} the constant could have been  made explicit at the expense of an even longer calculation.
	
	Just knowing the expectation and the variance can already give us some information about the variable. For example, we obtain the following proposition. This proposition shows that considering only a subset of the simplices already gives a good approximation for the critical simplex counts.
	We recall the notation that $f(n)= \omega (g(n))$ indicates that $\lim_{n \rightarrow \infty} \frac{f(n)}{g(n)} = \infty$. 
	
	\begin{proposition}\label{proposition:crit_lexi_uptoK}
		Fix $k \in [n]$. Let $K \leq n - k$ and set the random variable: \[T_{k+1}^K \coloneqq \sum_{\substack{s \in C_{k+1} \\ \min(s) \leq K}} \prod_{i \neq j \in s} Y_{i,j} \left[ \prod_{i=1}^{\min(s) - 1}\left(1 - \prod_{j \in s}Y_{i,j}\right) - \prod_{i=1}^{\min(s) - 1}\left(  1 - \prod_{j \in s_{-}} Y_{i,j}\right)  \right]  .\]
		
		If $K =K(n)= \omega(\ln^{1 + \epsilon}(n))$ for any $\epsilon > 0$, then the variable $T_{k+1} - T_{k+1}^K$ vanishes with high probability, provided that $p$ and $k$ stay constant. 
	\end{proposition}
	
	\begin{proof}
		A similar calculation to that for Lemma \ref{lemma:expect_crit_lexi} shows that: 
		\begin{align*}
		\EEbracket{T_{k+1} - T_{k+1}^K} &= \sum_{i=K+1}^{n-k} \binom{n-i}{k}p^{\binom{k+1}{2}} \left\{ (1-p^{k+1})^{i-1} - (1-p^k)^{i-1} \right\} \\
		&\leq \binom{n}{k}p^{\binom{k+1}{2}}(1-p^{k+1})^K\sum_{i=0}^{\infty} (1-p^{k+1})^{i} \leq p^{\binom{k+1}{2}-k-1}\frac{n^k}{k!}(1-p^{k+1})^K.
		\end{align*}
		
		Using Markov's inequality, we get:
		\[\PP(T_{k+1} - T_{k+1}^K \geq 1) \leq p^{\binom{k+1}{2}-k-1}\frac{n^k}{k!}(1-p^{k+1})^K,\]
		which asymptotically vanishes as long as $K = \omega(\ln^{1+\epsilon}(n))$. 
	\end{proof}
	
	\subsection{Approximation theorem}
	For $i \in [d]$, recall a random variable counting $i$-simplices in $\xnp$ that are critical under the lexicographical matching, as given in \eqref{equation:def_lexi_crit}. We write for the $i$-th index set $\I_i \coloneqq C_{i+1} \times \set{i}$. For $s = (\phi,i) \in \I_i$ we write \[\mu_{s} = p^{\binom{i+1}{2}}\left((1-p^{i+1})^{\min(\phi)-1}-(1-p^{i})^{\min(\phi)-1}\right)\] and $\sigma_i = \sqrt{\Var(T_{i+1})}$. Let \[X_{s} = \sigma_i^{-1}\left\{ \prod_{i \neq j \in \phi} Y_{i,j} \left[ \prod_{i=1}^{\min(\phi) - 1}\left(1 - \prod_{j \in \phi}Y_{i,j}\right) - \prod_{i=1}^{\min(\phi) - 1}\left(  1 - \prod_{j \in \phi_{-}} Y_{i,j}\right)  \right]  - \mu_{s}\right\}.\]
	
	Let $W_i = \sum_{s \in \I_i} X_{s}$ and $W = (W_1, W_2, \ldots, W_d) \in \RR^d$. For bounds that asymptotically go to zero for this example, we use Theorems \ref{theorem:mvn_dissociated_decomp_approx} and \ref{theorem:mvn_dissociated_decomp_approx_convex_sets} directly:  the uniform bounds from Corollary \ref{corollary:mvn_dissociated_decomp_approx} are not fine enough here.
	
	\begin{theorem}\label{theorem:crit_lexi_approx}
		Let $Z \sim \mvn{0}{\text{\rm Id}_{d \times d}}$ and $\Sigma$ be the covariance matrix of $W$.
		\begin{enumerate}
			\item Let $h \in \testzero$. Then there is a constant $B_{\ref{theorem:crit_lexi_approx}.1} >0$ independent of $n$ and a natural number $N_{\ref{theorem:crit_lexi_approx}.1}$ such that for any $n \geq N_{\ref{theorem:crit_lexi_approx}.1}$ we have\[\abs{\EE h(W) - \EE h(\Sigma^{\frac{1}{2}}Z)} \leq B_{\ref{theorem:crit_lexi_approx}.1} \abs{h}_3 n^{-1}.\]
			\item Let $\mathcal{K}$ be the class of convex sets in $\RR^d$. Then there is a constant $B_{\ref{theorem:crit_lexi_approx}.2} >0$ independent of $n$ and a natural number $N_{\ref{theorem:crit_lexi_approx}.2}$ such that for any $n \geq N_{\ref{theorem:crit_lexi_approx}.2}$ we have\[\sup _{A \in \mathcal{K}}|\PP(W \in A) -\PP(\Sigma^{\frac{1}{2}}Z \in A)| \leq B_{\ref{theorem:crit_lexi_approx}.2}n^{-\frac{1}{4}}.\]
		\end{enumerate}
	\end{theorem}

	\begin{proof}  It is clear that $W$ satisfies the conditions of Theorems \ref{theorem:mvn_dissociated_decomp_approx} and \ref{theorem:mvn_dissociated_decomp_approx_convex_sets} 
		for any $s = (\phi,i) \in \I_i$ setting
		\[\D_j(s) = \setc{(\psi, j) \in \I_j}{ |\phi \cap \psi| \geq 1}.\]
		We apply Theorems \ref{theorem:mvn_dissociated_decomp_approx} and \ref{theorem:mvn_dissociated_decomp_approx_convex_sets}. For the bounds on the quantity $B_{\ref{theorem:mvn_dissociated_decomp_approx}}$ from Theorems \ref{theorem:mvn_dissociated_decomp_approx} and \ref{theorem:mvn_dissociated_decomp_approx_convex_sets} we use Lemma \ref{lemma:unified_moments} and Lemma \ref{lemma:lower_bound_var_crit_lexi}. We write $C$ for an unspecified positive constant that does not depend on $n$. Also, we assume here that $n$ is large enough for the bound in Lemma \ref{lemma:lower_bound_var_crit_lexi} to apply. Let $\mu(i,a) = p^{\binom{i+1}{2}}\left( (1-p^{i+1})^{a - 1}-(1-p^{i})^{a - 1} \right).$ Then we have:
		\begin{align*}
		&B_{\ref{theorem:mvn_dissociated_decomp_approx}} \leq \frac{1}{3} \sum_{i,j,k=1}^d \sum_{a=1}^{n-i} \sum_{\substack{\phi \in C_{i+1} \\ \min(\phi) = a}} \sum_{b=1}^{n-j} \sum_{\substack{(\psi,j) \in \D_j((\phi,i)) \\ \min(\psi) = b}} \\
		&\Bigg\{ \sum_{r \in \D_k((\phi,i))} \frac{3}{2}  (\sigma_i \sigma_j \sigma_k)^{-1} \left\{ \mu(i,a)\mu(j,b)(1-\mu(i,a))(1-\mu(j,b)) \right\}^{\frac{1}{2}}  \\
		&+ \sum_{r \in \D_k((\psi,j))} (\sigma_i \sigma_j \sigma_k)^{-1} \left\{ \mu(i,a)\mu(j,b)(1-\mu(i,a))(1-\mu(j,b)) \right\}^{\frac{1}{2}} \Bigg\}\\
		&\leq  \sum_{i,j,k=1}^d \sum_{a=1}^{n-i} \sum_{b=1}^{n-j} Cn^{i+j-1} 
		n^k n^{-i-j-k} \left\{ (1-p^{i+1})^{a - 1}(1-p^{j+1})^{b - 1} + (1-p^{i})^{a - 1}(1-p^{j})^{b - 1} \right\}^{\frac{1}{2}} 
		\\
		&\leq Cn^{-1} \sum_{i,j,k=1}^d \sum_{a=1}^{\infty} \sum_{b=1}^{\infty} \left[\left\{ (1-p^{i+1})^{a - 1}(1-p^{j+1})^{b - 1} \right\}^{\frac{1}{2}} + \left\{ (1-p^{i})^{a - 1}(1-p^{j})^{b - 1} \right\}^{\frac{1}{2}} \right] \\
		&\leq Cn^{-1}  d^3 \left\{ \frac{1}{(1 - \sqrt{1-p^{d+1}})^2} + \frac{1}{(1 - \sqrt{1-p^{d}})^2} \right\} \leq  C  n^{-1}. \\
		\end{align*}
		
	\end{proof}
	
	\begin{remark}\label{remark:TDA}
		The relevance of understanding the number of critical simplices in the context of applied and computational topology is as follows. We assume that $p \in (0,1)$ and $k \in \{1, 2, \ldots\} $ are constants.
		
		\begin{enumerate}
			\item As seen in Lemma \ref{lemma:expect_crit_lexi}, the expected number of critical $k$-simplices under the lexicographical matching is one power of $n$ smaller than the total number of $k$-simplices in $\xnp$.
			\item In light of our approximation Theorem \ref{theorem:crit_lexi_approx} we also know that the (rescaled) deviations from the mean are approximately normal and the bounds are of the same order of $n$ compared to the approximation of all simplex counts in $\xnp$ as given in Theorem \ref{theorem:face_count_approx}. Knowing the expectation and the variance from Lemmas \ref{lemma:var_crit_lexi} and \ref{lemma:expect_crit_lexi}, one can apply concentration inequalities, for example, Chebyshev's inequality, to show that the number of critical simplices concentrates around its mean. Hence, because of the concentration of measure, the computational improvements as a result of lexicographical matching are likely substantial in $\xnp$.
			\item From Proposition \ref{proposition:crit_lexi_uptoK}, it is very likely that in $\xnp$ all $k$-simplices $s \in \xnp$ with $\min(s) = \omega(\ln^{1+\epsilon}(n))$ for any fixed $\epsilon > 0$ are not critical. 
		\end{enumerate}
		
	\end{remark}
	
	\section{Simplex Counts in Links}\label{section:link}
	Consider a random simplicial complex $\xnp$. For $1 \leq i < j \leq n$ define the edge indicator $Y_{i,j} \coloneqq \ind{\set{i,j} \in \xnp}$. In this section we study the count of $(k-1)$-simplices that would be in the link of a fixed subset $t \subseteq [n]$ if the subset spanned a simplex in $\xnp$. Given that $t$ is a simplex, the variable counts the number of $(k-1)$-simplices in $\lk(t)$. Thus, the random variable of interest is
	\begin{equation}\label{equation:def_link_count}
	T^t_k = \sum_{s \in C_k} \left\{ \ind{t \cap s = \varnothing} \prod_{i \neq j \in s} Y_{i,j} \prod_{i \in s, j \in t} Y_{i,j} \right\}.
	\end{equation}
	Note that the product $ \prod_{i \in s, j \in t} Y_{i,j}  $ ensures that $t \cup s$ is a simplex if $t$ spans a simplex. 
	
	\begin{remark}\label{remark:link_counts}
		The random variable $T^t_k$ does not fit into the framework of generalised $U$-statistics, which we will discuss in Section \ref{section:gen_u_stat}, because the summands depend not only on the variables that are indexed by the subset $s$ and we do not sum over all subsets $s$ but rather only the ones that do not intersect $t$. 
		
		Moreover, note that given the  number of vertices of the link of a simplex $t$, the conditional distribution of the link of $t$ is again $\xnp[n'][p]$, where $n'$ is a random variable equal to the number of vertices in the link. If we are interested in such a conditional distribution, the results of Section \ref{section:gen_u_stat} apply. However, in this section we study the number of simplices in the link of $t$ given that $t$ is a simplex rather than given the number of vertices of the link of $t$. Such a random variable behaves differently from the simplex counts in $\xnp$ which are studied in Section \ref{section:gen_u_stat}. For example, the summands of $T_k^t$ have a different dependence structure compared to the summands of $T_k$ from Equation \ref{tk}. As a result, the approximation bounds are of different order.
		
		 It is natural to ask whether the results obtained in this section follow from those of Section \ref{section:gen_u_stat} below. This might well be the case, but the answer is not straightforward. One could derive an approximation for the number of simplices in $\lk(t)$ given the number of vertices in the link; the variable $T_k^t$ could then be approximated by a mixture, induced by the distribution of the number of vertices in the link (which is binomial). However, applying this approach na\"ively yields bounds that do not converge to zero. While it is certainly possible that a different approach would succeed, we prefer not to rely on Section \ref{section:gen_u_stat} and prove the approximation directly.
	\end{remark}

	\subsection{Moments}
	It is easy to see that for any positive integer $k$ and $t \subseteq [n]$, 
	\[
	\EE \{ T^t_{k+1} \} = \binom{n - |t|}{k+1} p^{\binom{k+1}{2} + |t|(k+1)} =: \binom{n - |t|}{k+1} \mu_{k+1}^t
	\]
	since there are $\binom{n - |t|}{k+1}$ choices for $s \in C_{k+1}$ such that $s \cap t  = \varnothing$. Next we derive a lower bound on the variance.
	
	\begin{lemma}\label{lemma:variance_link_counts}
		For any fixed $1 \leq k \leq n-1$ and $t \subseteq [n]$ we have: 
		\[\Var(T^t_{k+1}) \geq
		(k+1)\binom{n-|t|}{2k+1}  \binom{2k+1}{k} (\mu_{k+1}^t)^2 \left\{p^{-|t|} - 1\right\}.\]
	\end{lemma}
	
	\begin{proof}
		First let us calculate $\Cov(T^t_{k+1}, T^t_{l+1})$. For fixed subsets $s \in C_{k+1}$ and $u \in C_{l+1}$ if $|s \cap u| = 0$, then the corresponding variables $\prod_{i \neq j \in s} Y_{i,j} \prod_{i \in s, j \in t} Y_{i,j}$ and $\prod_{i \neq j \in u} Y_{i,j} \prod_{i \in u, j \in t} Y_{i,j}$ are independent and so have zero covariance.
		
		For $1 \leq m \leq l + 1$, the number of pairs of subsets $s \in C_{k+1}$ and $u \in C_{l+1}$ such that $s \cap t = \varnothing = u \cap t$ and $|s \cap u| = m$ is $\binom{n-|t|}{k+1} \binom{k+1}{m} \binom{n-|t| - k - 1}{l + 1-m}$. Since each summand is non-negative, we lower bound by the $m=1$ summand and get (with $\binom{1}{2} := 0$) 
		\begin{align*}
		&\Cov(T^t_{k+1}, T^t_{l+1}) \\
		&= \sum_{m=1}^{l+1} \binom{n-|t|}{k+1} \binom{k+1}{m} \binom{n-|t| - k - 1}{l + 1-m}  \Big\{ \mu_{k+1}^t\mu_{l+1}^t  p^{- \binom{m}{2}} p^{-|t|m} -\mu_{k+1}^t\mu_{l+1}^t  \Big\} \\
		&\geq \binom{n-|t|}{k+1} (k+1) \binom{n-|t| - k - 1}{l} \mu_{k+1}^t\mu_{l+1}^t  \left\{p^{-|t|} - 1\right\} \\
		&= (k+1)\binom{n-|t|}{l+k+1}  \binom{l+k+1}{l} \mu_{k+1}^t\mu_{l+1}^t \left\{p^{-|t|} - 1\right\} .
		\end{align*}
		
		Taking $l = k$ completes the proof.
	\end{proof}
	
	\subsection{Approximation theorem}
	For a multivariate normal approximation of counts given in Equation \eqref{equation:def_link_count}, we write $\sigma_i = \sqrt{\Var(T_{i+1})}$ and $C_{i+1}^t = \setc{\phi \in C_{i+1}}{\phi \cap t = \varnothing}$, as well as $\I_i \coloneqq C_{i+1}^t \times \set{i}$. For $s = (\phi,i) \in \I_i$ define $$X_{s} = \sigma_i^{-1}(\prod_{i \neq j \in \phi} Y_{i,j} \prod_{i \in \phi, j \in t} Y_{i,j} - \mu_{i+1}^t).$$ It is clear that $\EEbracket{X_{s}} = 0$.
	Let $W^t_i = \sum_{s \in \I_i} X_{s}$ and $W^t = (W^t_1, W^t_2, \ldots, W^t_d) \in \RR^d$. Then we have the following approximation theorem. 
	
	\begin{theorem}\label{theorem:link_count_approx}
		Let $Z \sim \mvn{0}{\text{\rm Id}_{d \times d}}$ and $\Sigma$ be the covariance matrix of $W^t$. 
		\begin{enumerate}
			\item Let $h \in \testzero$. Then \[ \abs{\EE h(W^t) - \EE h(\Sigma^{\frac{1}{2}}Z)} \leq \abs{h}_3 B_{\ref{theorem:link_count_approx}} (n - |t|)^{-\frac{1}{2}}. \]
			\item Let $\mathcal{K}$ be the class of convex sets in $\RR^d$. Then \[\sup _{A \in \mathcal{K}}|\PP(W^t \in A) -\PP(\Sigma^{\frac{1}{2}}Z \in A)| \leq 2^{\frac{7}{2}} 3^{-\frac{3}{4}}d^{\frac{3}{16}}B_{\ref{theorem:link_count_approx}}^{\frac{1}{4}}(n - |t|)^{-\frac{1}{8}}.\]
		\end{enumerate}
		Here \[B_{\ref{theorem:link_count_approx}} = \frac{7}{6}  (2d+1)^{5d+\frac{17}{2}} (p^{-|t|} - 1)^{-\frac{3}{2}}p^{-(d+1)(d+2|t|)}.\]
		
	\end{theorem}
	
	\begin{proof}
		It is clear that $W^t$ satisfies the conditions of Corollary \ref{corollary:mvn_dissociated_decomp_approx} with the dependency neighbourhood $\D_j(s) = \setc{(\psi,j) \in \I_j}{ |\phi \cap \psi| \geq 1}$ for any $s = (\phi,i) \in \I_i$. So we aim to bound the quantity $B_{\ref{corollary:mvn_dissociated_decomp_approx}}$ from the corollary.
		
		Given $\phi \in C^t_{i+1}$ and $m \le \min (i+1, j+1)$ there are $\binom{i+1}{m}\binom{n-|t| - i-1}{j+1-m}$ subsets $\psi \in C^t_{j+1}$ such that $|\phi \cap \psi| = m$. Therefore, for any $i,j \in [d]$ and $s \in \I_i$ we have 
		\begin{eqnarray}\label{equation:link1}
		|\D_j(s)| &=& \sum_{m=1}^{\min(i,j) + 1} \binom{i+1}{m} \binom{n-|t|-i-1}{j+1-m} \nonumber \\
		&\leq & (i+1)^{\min(i,j)+2} (n-|t|)^{j} \nonumber \\
		&\leq & (d+1)^{d+2} (n-|t|)^{j}
		\end{eqnarray}
		giving a bound for $\alpha_{ij}$. For a bound on $\beta_{ijk}$, applying Lemma \ref{lemma:unified_moments}, for any $i,j,k \in [d]$ and $s \in \I_i, u \in \I_j, v \in \I_k$ we get 
		\begin{align}\label{equation:link2}
		&\EE \abs{X_{s} X_{u} X_{v}} \leq (\sigma_i \sigma_j \sigma_k)^{-1} \left\{ \mu_{i+1}^t\mu_{j+1}^t(1-\mu_{i+1}^t)(1-\mu_{j+1}^t) \right\}^{\frac{1}{2}}; \\
		&\EE \abs{X_{s} X_{u}} \EE \abs{X_{v}} \leq (\sigma_i \sigma_j \sigma_k)^{-1}  \left\{ \mu_{i+1}^t\mu_{j+1}^t(1-\mu_{i+1}^t)(1-\mu_{j+1}^t) \right\}^{\frac{1}{2}}.
		\end{align}
		Now we apply Corollary \ref{lemma:variance_link_counts} and get
		\begin{align*}
		\sigma_i^2 &\geq (i+1)\binom{n-|t|}{2i+1}  \binom{2i+1}{i} (\mu_{k+1}^t)^2 \left\{p^{-|t|} - 1\right\}  \\
		&\geq \frac{(n-|t|)^{2i+1}}{(2d+1)^{d+1}d^{d}} (\mu_{k+1}^t)^2 \left\{p^{-|t|} - 1\right\} .
		\end{align*}
		Taking both sides of the inequality to the power of $-\frac{1}{2}$ we get for any $i \in [d]$
		\begin{equation}\label{equation:link3}
		\sigma_i^{-1} \leq (n-|t|)^{-i - \frac{1}{2}}(2d+1)^{\frac{d+1}{2}}d^{\frac{d}{2}} (\mu_{k+1}^t)^{-1}\left\{p^{-|t|} - 1\right\}^{-\frac{1}{2}}.
		\end{equation}
		
		Using Equations \eqref{equation:link1} - \eqref{equation:link3} to bound $B_{\ref{corollary:mvn_dissociated_decomp_approx}}$ from Corollary \ref{corollary:mvn_dissociated_decomp_approx} we get:
		\begin{align*}
		& B_{\ref{corollary:mvn_dissociated_decomp_approx}}
		\leq \frac{7}{6} \sum_{i,j,k=1}^d \binom{n-|t|}{i+1} (d+1)^{2d+4} (n-|t|)^{j+k} (\sigma_i \sigma_j \sigma_k)^{-1} \\
		&\left\{ \mu_{i+1}^t\mu_{j+1}^t(1-\mu_{i+1}^t)(1-\mu_{j+1}^t) \right\}^{\frac{1}{2}} \\
		\leq &\frac{7}{6}  \sum_{i,j,k=1}^d (n-|t|)^{i+j+k+1} (d+1)^{2d+4} (n-|t|)^{-i -j -k - \frac{3}{2}} (2d+1)^{\frac{3d+3}{2}}d^{\frac{3d}{2}} \\
		&(p^{-|t|} - 1)^{-\frac{3}{2}}(\mu_{k+1}^t\mu_{i+1}^t\mu_{j+1}^t)^{-1} \left\{ \mu_{i+1}^t\mu_{j+1}^t(1-\mu_{i+1}^t)(1-\mu_{j+1}^t) \right\}^{\frac{1}{2}} \\
		\leq & (n - |t|)^{-\frac{1}{2}} \frac{7}{6}  (2d+1)^{5d+\frac{11}{2}} (p^{-|t|} - 1)^{-\frac{3}{2}} \sum_{i,j,k=1}^d \left((\mu_{i+1}^t\mu_{j+1}^t)^{-1}(\mu_{k+1}^t)^{-2}\right)^{\frac{1}{2}} \\
		\leq &\left\{\frac{7}{6}  (2d+1)^{5d+\frac{17}{2}} (p^{-|t|} - 1)^{-\frac{3}{2}}p^{-(d+1)(d+2|t|)} \right\} (n - |t|)^{-\frac{1}{2}}.
		\end{align*}
	\end{proof}

	\begin{remark}\label{remark:size_of_link}
		Recall that $\EEbracket{T^t_{k+1}} = \binom{n - |t|}{k+1}p^{\binom{k+1}{2} + |t|(k+1)}$. By Stirling's approximation, if $p \in (0,1)$ is a constant,  then $\max(k, |t|) = \Omega(\ln^{1+\epsilon}(n))$ for any positive $\epsilon$ forces the expectation to go to $0$ asymptotically. Hence, by Markov's inequality, with high probability there are no $k$-simplices in the link of $t$ as long as $\max(k, |t|)$ is of order $\ln^{1+\epsilon}(n)$ or larger for any $\epsilon > 0$ for a constant $p$.
		
		Recall that in Theorem \ref{theorem:link_count_approx} we count all simplices up to dimension $d$ in the link of $t$. Note that if $\max(d^2, d|t|) = O(\ln^{1 - \epsilon}(n))$ for any $\epsilon > 0$, then the bounds in Theorem \ref{theorem:link_count_approx} tend to $0$ as $n$ tends to infinity as long as $p \in (0,1)$ stays constant. In particular, if $d$ is a constant, Theorem \ref{theorem:link_count_approx} gives an approximation for all sizes of $t$ for which the approximation is needed.
	\end{remark}

	\section{Simplex Counts in \texorpdfstring{$\xnp$}{X(n,p)}}\label{section:gen_u_stat}
	
	In this section we study the simplex counts in $\xnp$ or, equivalently, the clique counts in $\gnp$. In order to do that, we prove a multivariate normal approximation theorem for generalised $U$-statistics, which might be of independent interest. The approximation theorem for simplex counts in $\xnp$ then follows as a special case. 
	
	Here we consider generalised $U$-statistics, which were first introduced in \cite{janson1991asymptotic}. We expand the notion slightly by considering independent but not necessarily identically distributed variables instead of i.i.d.\,variables.
	
	Let $\{\xi_i\}_{1 \leq i \leq n}$  be a sequence of of independent random variables taking values in a measurable set  $\mathcal{X}$ and  let $\{Y_{i,j}\}_{1 \leq i < j \leq n}$ be an array of of independent random variables taking values in a measurable set $\mathcal{Y}$ which is independent of $\{\xi_i\}_{1 \leq i \leq n}$. We use the convention that $Y_{i,j} = Y_{j,i}$ for any $i < j$. For example, one can think of $X_i$ as a random label of a vertex $i$ in a random graph where $Y_{i,j}$ is the indicator for the edge connecting $i$ and $j$. Given a subset $s \subseteq [n]$ of size $m$, write $s = \set{s_1, s_2, \ldots, s_m}$  such that $s_1 < s_2 < \ldots < s_m$ and set $\mathcal{X}_s = (\xi_{s_1}, \xi_{s_2}, \ldots, \xi_{s_m})$ and $\mathcal{Y}_s = (Y_{s_1, s_2}, Y_{s_1, s_3}, \ldots Y_{s_{m-1},s_m})$. Recall that $C_k$ denotes the set of subsets of $[n]$ which are of size $k$.
	
	\begin{definition}
		\label{def:gen_u_statistic}
		Given $1 \leq k \leq n$ and a measurable function $f: \mathcal{X}^k \times \mathcal{Y}^k \to \RR$ define the associated generalised $U$-statistic by
		\[S_{n,k}(f) = \sum_{s \in C_k} f(\mathcal{X}_s, \mathcal{Y}_s).\]
	\end{definition}
	
	\subsection{The First Approximation Theorem}\label{subsection:approximation_ustat}
	
	Let $\set{k_i}_{i \in [d]}$ be a collection of positive integers, each being at most $n$, and for each $i \in [d]$ let $f_i: \mathcal{X}^{k_i} \times \mathcal{Y}^{k_i} \to \RR$ be a measurable function. We are interested in the joint distribution of $S_{n,k_1}(f_1), S_{n,k_2}(f_2), \ldots S_{n,k_d}(f_d)$.
	
	Fix $i \in [d]$. For $s \in \I_i \coloneqq C_{k_i} \times \set{i}$ define $X_{s} = \sigma_i^{-1}\left( f_i(\mathcal{X}_s, \mathcal{Y}_s) -  \mu_{s}\right)$, where $\mu_{s} = \EEbracket{f_i(\mathcal{X}_s, \mathcal{Y}_s)}$ and $\sigma_i^{2} = \Var(S_{n,k_i}(f_i))$. Now let $W_i = \sum_{s \in \I_i} X_{s}$ be a random variable and write $W = (W_1, W_2, \ldots W_d) \in \RR^d$.  By construction, $W_i$ has mean 0 and variance 1.
	
	\begin{assumption}\label{assumption:gen_u_stat}
		We assume that
		\begin{enumerate}
			\item For any $i \in [d]$ there is some $\alpha_i > 0$ such that for all $s, t \in \I_i$, the variables $f_i(\mathcal{X}_s, \mathcal{Y}_s), f_i(\mathcal{X}_t, \mathcal{Y}_t)$ are either independent or  $\Cov(f_i(\mathcal{X}_s, \mathcal{Y}_s), f_i(\mathcal{X}_t, \mathcal{Y}_t)) > \alpha_i$.
			\item There is $\beta \geq 0$ such that for any $i,j,l \in [d]$ and any $s \in \I_i, t \in \I_j, u \in \I_l$ we have \[\EE \abs{\left\{ f_i(\mathcal{X}_s, \mathcal{Y}_s) - \mu_{s} \right\}\left\{ f_j(\mathcal{X}_t, \mathcal{Y}_t) - \mu_{t} \right\}\left\{ f_l(\mathcal{X}_u, \mathcal{Y}_u) - \mu_{u} \right\}} \leq \beta\] as well as \[\EE \abs{\left\{ f_i(\mathcal{X}_s, \mathcal{Y}_s) - \mu_{s} \right\}\left\{ f_j(\mathcal{X}_t, \mathcal{Y}_t) - \mu_{t} \right\}} \EE \abs{ f_l(\mathcal{X}_u, \mathcal{Y}_u) - \mu_{u}} \leq \beta.\]
			\item The random variables $X_s$ have finite absolute third moments.
		\end{enumerate}
	\end{assumption}
	
	The first assumption is not necessary but very convenient and we use it to derive a lower bound for the variance $\sigma_i^2$. It holds in a variety of settings, for example, subgraph counts in a random graph. A normal approximation theorem can be proven in our framework when the assumption does not hold and a sufficiently large lower bound for the variance is acquired in a different way. Similarly, we use the second assumption to get a convenient bound on mixed moments. However, depending on a particular question at hand, one might want to use a bound on mixed moments, which is not uniform in $i,j,l$ and sometimes even one that is not uniform in $s,u,v$. We will discuss such an example (which does not fit into the framework of generalised $U$-statistics) in Section \ref{section:crit_lexi}. In this section, in order to maintain the generality and simplicity of the proofs, we work under Assumption \ref{assumption:gen_u_stat}. In \cite[Theorem 6]{janson1991asymptotic} it is assumed that all summands in the generalised $U$-statistic have finite second moment as well as that the sums admit a particular decomposition which is not easily translatable to our framework. In contrast to  \cite{janson1991asymptotic} we obtain a non-asymptotic bound on the normal approximation, as follows.
	
	\begin{theorem}\label{theorem:ustat_approx}
		Let $Z \sim \mvn{0}{\text{\rm Id}_{d \times d}}$ and let $W$ with covariance matrix $\Sigma$ satisfy Assumption \ref{assumption:gen_u_stat}.
		\begin{enumerate}
			\item Let $h \in \testzero$. Then 
			\[
			\abs{\EE h(W) - \EE h(\Sigma^{\frac{1}{2}}Z)} \leq \abs{h}_3  B_{\ref{theorem:ustat_approx}}  n^{-\frac{1}{2}}.
			\]
			
			\item Let $\mathcal{K}$  be a class of convex sets in $\RR^d$. Then 
			\[ \sup _{A \in \mathcal{K}}|\PP(W \in A)-\PP(\Sigma^{\frac{1}{2}}Z \in A)| \leq 2^{\frac{7}{2}} 3^{-\frac{3}{4}}d^{\frac{3}{16}}B_{\ref{theorem:ustat_approx}}^{\frac{1}{4}}n^{-\frac{1}{8}}. \]
		\end{enumerate}
		Here,
		\[
		B_{\ref{theorem:ustat_approx}} = \frac{2 \psi}{3} \sum_{i,j,l=1}^d \frac{k_i^{\min(k_i,k_j) + 1}}{k_i! \sqrt{\alpha_i \alpha_j \alpha_l}} \left(  k_i^{\min(k_i,k_l) + 1} + k_j^{\min(k_j,k_l) + 1} \right) K_i K_j K_l
		\]
		and
		\[
		K_i = (2k_i^2 - k_i)^{-\frac{k_i}{2} + \frac{1}{2}}.
		\]
	\end{theorem}
	
	\begin{proof}
		Note that if $s = (\phi,i) \in \I_i$ and $u = (\psi, j) \in \I_j$ are chosen such that $\phi \cap \psi = \varnothing$, then the corresponding variables $X_{s}$ and $X_{u}$ are independent since $f_i(\mathcal{X}_s, \mathcal{Y}_s)$ and $f_j(\mathcal{X}_u, \mathcal{Y}_u)$ do not share any random variables from the sets $\{\xi_i\}_{1 \leq i \leq n}$ and $\{Y_{i,j}\}_{1 \leq i < j \leq n}$. Hence, if for any $s = (\phi,i) \in \I_i$ we set $\D_j(s) = \setc{(\psi, j) \in \I_j}{|\phi \cap \psi| \geq 1}$, then $W$ satisfies the assumptions of Corollary \ref{corollary:mvn_dissociated_decomp_approx}. It remains to bound the quantity $B_{\ref{corollary:mvn_dissociated_decomp_approx}}$.
		
		First, to find $\alpha_{ij}$ as in Corollary \ref{corollary:mvn_dissociated_decomp_approx}, given $\phi \in C_{k_i}$ and if $k_i, k_j \ge m$ then there are $\binom{k_i}{m}\binom{n-k_i}{k_j-m}$ subsets $\psi \in C_{k_j}$ such that $|\phi \cap \psi| = m$. Therefore, we have for any $i,j \in [d]$ and $s \in \I_i$
		\begin{equation}\label{equation:ustat1}
		|\D_j(s)| = \sum_{m=1}^{\min(k_i,k_j)} \binom{k_i}{m}\binom{n-k_i}{k_j-m}  = \alpha_{ij} \leq k_i^{\min(k_i,k_j) + 1} (n-k_i)^{k_j-1}.
		\end{equation}
		Note that
		\[\EE \abs{X_{s} X_{t} X_{u}} = (\sigma_i \sigma_j \sigma_k)^{-1} \EE \abs{\left\{ f_i(\mathcal{X}_s, \mathcal{Y}_s) - \mu_{s} \right\}\left\{ f_j(\mathcal{X}_t, \mathcal{Y}_t) - \mu_{t} \right\}\left\{ f_l(\mathcal{X}_u, \mathcal{Y}_u) - \mu_{u} \right\}}\]
		as well as 
		\[\EE \abs{X_{s} X_{t}} \EE \abs{X_{u}} = (\sigma_i \sigma_j \sigma_k)^{-1} \EE \abs{\left\{ f_i(\mathcal{X}_s, \mathcal{Y}_s) - \mu_{s} \right\}\left\{ f_j(\mathcal{X}_t, \mathcal{Y}_t) - \mu_{t} \right\}} \EE \abs{ f_l(\mathcal{X}_u, \mathcal{Y}_u) - \mu_{u}}.\]
		Using Assumption \ref{assumption:gen_u_stat},
		for any $i,j,l \in [d]$ and $s \in \I_i, t \in \I_j, u \in \I_l$
		\begin{equation}\label{equation:ustat2}
		\EE \abs{X_{s} X_{t} X_{u}} \leq (\sigma_i \sigma_j \sigma_k)^{-1} \beta \; \text{ and } \;
		\EE \abs{X_{s} X_{t}} \EE \abs{X_{u}}\leq (\sigma_i \sigma_j \sigma_k)^{-1} \beta.
		\end{equation}
		To take care of the variance terms, we lower bound the variance using Assumption \ref{assumption:gen_u_stat};
		\begin{align*}
		\Var(S_{n,k_i}(f_i)) &= \sum_{s \in C_{k_i}} \sum_{t \in \D_i(s)} \Cov(f_i(\mathcal{X}_s, \mathcal{Y}_s), f_i(\mathcal{X}_t, \mathcal{Y}_t)) \\
		&= \sum_{m=1}^{k_i}\sum_{s \in C_{k_i}} \sum_{\substack{t \in  C_{k_i} \\ |s \cap t| = m}} \Cov(f_i(\mathcal{X}_s, \mathcal{Y}_s), f_i(\mathcal{X}_t, \mathcal{Y}_t)) \\
		&\geq \binom{n}{k_i} \sum_{m=1}^{k_i} \binom{k_i}{m}\binom{n-k_i}{k_i-m} \alpha_i = \alpha_i \sum_{m=1}^{k_i} \binom{n}{2k_i-m} \binom{2k_i-m}{k_i} \binom{k_i}{m} \\
		& \geq \alpha_i k_i \binom{n}{2k_i-1} \binom{2k_i-1}{k_i} \geq \alpha_i k_i \frac{n^{2k_i - 1}}{(2k_i - 1)^{2k_i - 1}} \frac{(2k_i-1)^{k_i}}{k_i^{k_i}} \\
		&= \alpha_i \frac{n^{2k_i - 1}}{(2k_i^2 - k_i)^{k_i - 1}}.
		\end{align*}
		Here the second-to-last inequality follows by taking only the term for $m=1$. Now we take both sides of the inequality to the power of $-\frac{1}{2}$ to get that for any $i \in [d]$
		\begin{equation}\label{equation:ustat3}
		\sigma_i^{-1} \leq n^{-k_i + \frac{1}{2}}\alpha_i^{-\frac{1}{2}}(2k_i^2 - k_i)^{-\frac{k_i}{2} + \frac{1}{2}}.
		\end{equation}
		Using Equations \eqref{equation:ustat1} - \eqref{equation:ustat3} to bound the quantity $B_{\ref{corollary:mvn_dissociated_decomp_approx}}$ from Corollary \ref{corollary:mvn_dissociated_decomp_approx} we get
		\begin{align*}
		&B_{\ref{corollary:mvn_dissociated_decomp_approx}} \leq \frac{2}{3}  \sum_{i,j,l=1}^d (\sigma_i \sigma_j \sigma_k)^{-1} \beta \binom{n}{k_i} k_i^{\min(k_i,k_j) + 1} (n-k_i)^{k_j-1} \\
		&\Bigg\{  k_i^{\min(k_i,k_l) + 1} (n-k_i)^{k_l-1} + k_j^{\min(k_j,k_l) + 1} (n-k_j)^{k_l-1} \Bigg\} \\
		\leq &\frac{2}{3} \sum_{i,j,l=1}^d n^{k_i + k_j + k_l - 2} \frac{k_i^{\min(k_i,k_j) + 1}}{k_i!} \Bigg\{  k_i^{\min(k_i,k_l) + 1} + k_j^{\min(k_j,k_l) + 1} \Bigg\} \beta \\
		& \left( n^{-k_i - k_j - k_l + \frac{3}{2}}(\alpha_i \alpha_j \alpha_l)^{-\frac{1}{2}}(2k_i^2 - k_i)^{-\frac{k_i}{2} + \frac{1}{2}} (2k_j^2 - k_j)^{-\frac{k_j}{2} + \frac{1}{2}} (2k_l^2 - k_l)^{-\frac{k_l}{2} + \frac{1}{2}}\right) \\
		\leq & \left\{ \frac{2 \beta}{3} \sum_{i,j,l=1}^d \frac{k_i^{\min(k_i,k_j) + 1}}{k_i! \sqrt{\alpha_i \alpha_j \alpha_l}} \left(  k_i^{\min(k_i,k_l) + 1} + k_j^{\min(k_j,k_l) + 1} \right) K_i K_j K_l \right\} n^{-\frac{1}{2}}.
		\end{align*}
	\end{proof}
	
	\subsection{Approximation Theorem with no Variables \texorpdfstring{$\set{\xi_i}_{i \in [n]}$}{xi\_i}}
	\label{subsection:approximation_ustat_no_x}
	
	Next we consider the special case that the functions in Definition \ref{def:gen_u_statistic} only depend on the second component, so that the sequence $\set{\xi_i}_{i \in [n]}$ can be ignored. Continuing to use the same notation, we want to understand the joint distribution of  $S_{n,k_1}(f_1), S_{n,k_2}(f_2), \ldots S_{n,k_d}(f_d)$. However, we add an additional assumption.
	\begin{assumption}\label{assumption:ustat_no_x}
		We assume that the functions $f_i$ only depend on the variables $\set{Y_{i,j}}$ for $1 \leq i < j \leq n$. That is, we can write $f_i: \mathcal{Y}^{k_i} \to \RR$.
	\end{assumption}
	Such functions appear naturally, for example, when counting subgraphs in an inhomogeneous Bernoulli random graph. A detailed example of such generalised $U$-statistic is worked out in Section \ref{sect:face_counts}.
	
	In this case, we can adapt the previous theorems slightly and get improved bounds. We still work under Assumption \ref{assumption:gen_u_stat}. The key difference in this case is that the dependency neighbourhoods become smaller: now the subsets need to overlap in at least 2 elements for the corresponding summands to share at least one variable $Y_{i,j}$ and hence become dependent. This makes both the variance and the size of dependency neighbourhoods smaller. In the context of Theorem \ref{theorem:mvn_dissociated_decomp_approx}, the trade-off works out in our favour to give smaller bounds, as follows. For any $s = (\phi,i) \in \I_i$ we set $\D_j(s) = \setc{(\psi,j) \in \I_j}{|\phi \cap \psi| \geq 2}$, so that $W$, under the additional Assumption \ref{assumption:ustat_no_x}, satisfies the assumptions of Corollary \ref{corollary:mvn_dissociated_decomp_approx}.
	
	In this case, we can adjust Equations \eqref{equation:ustat1} and \eqref{equation:ustat3}. The proofs are exactly the same as previously, with the only difference being that when we sum over $m$, we start at $m=2$ as opposed to $m=1$.
	
	\begin{theorem}\label{theorem:ustat_approx_no_x}
		Consider $W$ that satisfies Assumption \ref{assumption:ustat_no_x}. Let $Z \sim \mvn{0}{\text{\rm Id}_{d \times d}}$ and $\Sigma$ be the covariance matrix of $W$.
		\begin{enumerate}
			\item  Let $h \in \testzero$. Then  \[
			\abs{\EE h(W) - \EE h(\Sigma^{\frac{1}{2}}Z)} \leq \abs{h}_3  B_{\ref{theorem:ustat_approx_no_x}}  n^{-1};
			\]
			\item  Let $\mathcal{K}$  be a class of convex sets in $\RR^d$. Then  \[ \sup _{A \in \mathcal{K}}|\PP(W \in A)-\PP(\Sigma^{\frac{1}{2}}Z \in A)| \leq 2^{\frac{7}{2}} 3^{-\frac{3}{4}}d^{\frac{3}{16}}B_{\ref{theorem:ustat_approx_no_x}}^{\frac{1}{4}}n^{-\frac{1}{4}}. \]
		\end{enumerate}
		Here
		\[
		B_{\ref{theorem:ustat_approx_no_x}} = \frac{16 \beta}{3} \sum_{i,j,l=1}^d \frac{k_i^{\min(k_i,k_j) + 1}}{k_i! \sqrt{\alpha_i \alpha_j \alpha_l}} \left(  k_i^{\min(k_i,k_l) + 1} + k_j^{\min(k_j,k_l) + 1} \right) K_i K_j K_l
		\]
		and
		\[
		K_i = (2k_i^2 - k_i)^{-\frac{k_i}{2} + \frac{1}{2}}.
		\]
	\end{theorem}

	\begin{proof}
		Equation \eqref{equation:ustat1} becomes \[|\D_j(s)| \leq k_i^{\min(k_i,k_j) + 1} (n-k_i)^{k_j-2}.\]
		Equation \eqref{equation:ustat3} becomes \[ \sigma_i^{-1} \leq 2 n^{-k_i + 1}\alpha_i^{-\frac{1}{2}}(2k_i^2 - 2k_i)^{-\frac{k_i}{2} + 1}. \]
		Using the adjusted bounds in Corollary \ref{corollary:mvn_dissociated_decomp_approx} gives  the result.
	\end{proof}
	
	\subsection{Approximation Theorem for Simplex Counts}\label{sect:face_counts}
	
	In this section we apply Theorem \ref{theorem:ustat_approx_no_x} to approximate simplex counts. 
	Consider $G \sim \gnp$. For $1 \leq x < y \leq n$ let $Y_{x,y} \coloneqq \ind{x \sim y}$ be the edge indicator. In this section we are interested in the $(i+1)$-clique count in $\gnp$ or, equivalently, the  $i$-simplex count in $\xnp$, given by
	\begin{equation}\label{tk} 
	T_{i+1} = \sum_{s \in C_{i+1}} \prod_{x \neq y \in s} Y_{x,y}. 
	\end{equation}
	
	Let $\mathcal{Y}^{i+1} = \{0,1\}^{i+1}$ and let $f_i: \mathcal{Y}^{i+1} \to \RR$ be the function \[f_i(\mathcal{Y}_s) = {\prod_{Y_{x,y} \in \mathcal{Y}_s} Y_{x,y}}.\] Then the associated generalised U-statistic $S_{n, i + 1}(f_i)$ equals the $(i+1)$-clique count $T_{i+1}$, as given by Equation \eqref{tk}. To  apply Theorem \ref{theorem:ustat_approx_no_x} we need to center and rescale our variables. It is easy to see that $\EE \{ f_i(\mathcal{Y}_\phi)\} = p^{\binom{i+1}{2}}.$ if $\phi \in C_{i+1}$ Just like in Section \ref{subsection:approximation_ustat}, we let $I_i \coloneqq C_{i+1} \times \set{i}$ and for $s = (\phi,i) \in \I_i$ we define $X_{s} \coloneqq  {\sigma^{-1}} \left( f_i(\mathcal{Y}_\phi) - p^{\binom{i+1}{2}} \right) $ and $W_i = \sum_{s \in \I_i} X_{s}$. Now the vector of interest is $W = (W_1, W_2, \ldots, W_d) \in \RR^d$. This brings us to the next approximation theorem.
	
	\begin{corollary}\label{theorem:face_count_approx}
		Let $Z \sim \mvn{0}{\text{\rm Id}_{d \times d}}$ and $\Sigma$ be the covariance matrix of $W$.
		\begin{enumerate}
			\item   Let $h \in \testzero$. Then \[ \abs{\EE h(W) - \EE h(\Sigma^{\frac{1}{2}}Z)} \leq \abs{h}_3 B_{\ref{theorem:face_count_approx}} n^{-1}. \]
			\item  Let $\mathcal{K}$  be a class of convex sets in $\RR^d$. Then  \[ \sup _{A \in \mathcal{K}}|\PP(W \in A)-\PP(\Sigma^{\frac{1}{2}}Z \in A)| \leq 2^{\frac{7}{2}} 3^{-\frac{3}{4}}d^{\frac{3}{16}}B_{\ref{theorem:face_count_approx}}^{\frac{1}{4}}n^{-\frac{1}{4}}. \]
		\end{enumerate}
		Here \[B_{\ref{theorem:face_count_approx}} = \frac{16}{3} d^{2d+5} p^{-3\binom{d+1}{2} + 1}(1-p^{\binom{d+1}{2}})(p^{-1} - 1)^{-\frac{3}{2}}.\]
	\end{corollary}
	
	\begin{proof}
		Firstly, observe that for any $\phi, \psi \in C_{i+1}$ for which $| \phi \cap \psi| \leq 1$ the covariance vanishes, while if $| \phi \cap \psi| \geq 2$ the covariance is non-zero, and we have\[\Cov(f_i(\mathcal{Y}_\phi), f_i(\mathcal{Y}_\psi)) = p^{2\binom{i+1}{2} - \binom{|\phi \cap \psi|}{2}} - p^{2\binom{i+1}{2}} \geq p^{2\binom{i+1}{2}}(p^{-1} - 1).\]
		
		For $s = (\phi,i) \in \I_i$ write $\hat{X}_{s} = f_i(\mathcal{Y}_\phi) - p^{\binom{i+1}{2}}$. Then by Lemma \ref{lemma:unified_moments} we get:
		\begin{align*}
		&\EE \abs{\hat{X}_{s} \hat{X}_{t}} \EE \abs{\hat{X}_{u}} \leq  \left\{ p^{\binom{i+1}{2} + \binom{j+1}{2}}(1-p^{\binom{i+1}{2}})(1-p^{\binom{j+1}{2}}) \right\}^{\frac{1}{2}}; \\
		&\EE \abs{\hat{X}_{s} \hat{X}_{t} \hat{X}_{u}} \leq  \left\{ p^{\binom{i+1}{2} + \binom{j+1}{2}}(1-p^{\binom{i+1}{2}})(1-p^{\binom{j+1}{2}}) \right\}^{\frac{1}{2}}.
		\end{align*}
		
		Since $ \left\{ p^{\binom{i+1}{2} + \binom{j+1}{2}}(1-p^{\binom{i+1}{2}})(1-p^{\binom{j+1}{2}}) \right\}^{\frac{1}{2}} \leq p(1-p^{\binom{d+1}{2}})$, we see that Assumption \ref{assumption:gen_u_stat} holds. Assumption \ref{assumption:ustat_no_x} also holds and therefore we can apply Theorem \ref{theorem:ustat_approx_no_x} with $k_i = i+1$, $K_i = (2(i+1)^2 - 2(i+1))^{-\frac{1}{2}(i+1) + 1}$, $\alpha_i = p^{2\binom{i+1}{2}}(p^{-1} - 1)$, and $\beta = p(1-p^{\binom{d+1}{2}})$. Using the bounds $K_i \leq 1$ as well as $2 \leq k_i^{\min(k_i,k_j) + 1} \leq d^{d+1}$, and $\sqrt{\alpha_i} \geq p^{\binom{d+1}{2}}\sqrt{p^{-1}-1}$ finishes the proof.
	\end{proof}
	
	\begin{remark}\label{remark:size_of_cliques}
		It is easy to show that with high probability there are no large cliques in $\gnp$ for $p<1$ constant. To see this, the expectation of the number of $k$-cliques is $\binom{n}{k}p^{\binom{k}{2}}$. By Stirling's approximation, $k = \Omega(\ln^{1+\epsilon}(n))$ for any positive $\epsilon$ forces the expectation to go to $0$ asymptotically. Hence, by Markov's inequality, with high probability there are no cliques of order $\ln^{1+\epsilon}(n)$ or larger for any $\epsilon > 0$. For cliques of order larger than $\ln^{\frac{1}{2}}(n)$ and fixed $p$, a Poisson approximation might be more suitable.
		
		Recall that in Corollary \ref{theorem:face_count_approx} the size of the maximal clique we count is $d + 1$. Note that if $d = O(\ln^{\frac{1}{2} - \epsilon}(n))$ for any $\epsilon > 0$, then the bounds in Corollary \ref{theorem:face_count_approx} tend to $0$ as $n$ tends to infinity as long as $p \in (0,1)$ stays constant. This might seem quite small but in the light there not being any cliques of order $\ln^{1+\epsilon}(n)$ with high probability, this is meaningfully large. 
	\end{remark}
	
	\begin{remark}
		Note that in Corollary \ref{theorem:face_count_approx} we use multivariate normal distribution with covariance $\Sigma$, which is the covariance of $W$ when $n$ is finite and it differs from the limiting covariance, as mentioned in \cite{reinert2010random}. To approximate $W$ with the limiting distribution, one could proceed in the spirit of \cite[Proposition 3]{reinert2010random} in two steps: use the existing theorems to approximate $W$ with $\Sigma Z$ and then approximate $\Sigma Z$ with $\Sigma_L Z$ where $\Sigma_L$ is the limiting covariance, which is non-invertible, as observed in \cite{janson1991asymptotic}.
	\end{remark}
	
	\begin{remark}
		Corollary \ref{theorem:face_count_approx} generalises the result \cite[Proposition 2]{reinert2010random} beyond the case when $d=2$ and we get a bound of the same order of $n$. \cite[Theorem 3.1]{kaur2020higher} considers centered subgraph counts in a random graph associated to a graphon. If we take the graphon to be constant, the associated random graph is just $\gnp$. Compared to \cite[Theorem 3.1]{kaur2020higher} we place weaker smoothness conditions on our test functions. However, we make use of the special structure of cliques whereas \cite[Theorem 3.1]{kaur2020higher} applies to any centered  subgraph counts. Translating \cite[Theorem 3.1]{kaur2020higher} into a result for uncentered subgraph counts, as we provide here in the special case of clique counts, is not trivial for general $d$.
		
		However, it should be possible to extend our results, using the same abstract approximation theorem, beyond the random clique complex to Linial-Meshulam random complexes \cite{linial2006homological} or even more general multiparamter Costa-Farber random complexes \cite{costa2016large}. We shall consider this conjecture in future work.
	\end{remark}

	\printbibliography

@book{edelsbrunner_comp_top,
	title = {Computational Topology: An Introduction},
	author = {Edelsbrunner, Herbert and Harer, John},
	year = {2010},
	publisher = {American Mathematical Society},
	address = {Providence, USA}
}

@article {ghrist:08,
    AUTHOR = {Ghrist, Robert},
     TITLE = {Barcodes: the persistent topology of data},
   JOURNAL = {Bull. Amer. Math. Soc. (N.S.)},
  FJOURNAL = {American Mathematical Society. Bulletin. New Series},
    VOLUME = {45},
      YEAR = {2008},
    NUMBER = {1},
     PAGES = {61--75},
}

@article{tda2,
	title={Persistence barcodes for shapes},
	author={Carlsson, Gunnar and Zomorodian, Afra and Collins, Anne and Guibas, Leonidas J},
	journal={International Journal of Shape Modeling},
	volume={11},
	number={02},
	pages={149--187},
	year={2005},
	publisher={World Scientific}
}

@article{roadmap_persistence,
  title={A roadmap for the computation of persistent homology},
  author={Otter, Nina and Porter, Mason A and Tillmann, Ulrike and Grindrod, Peter and Harrington, Heather A},
  journal={EPJ Data Science},
  volume={6},
  pages={1--38},
  year={2017},
  publisher={Springer}
}

@article{forman2002user,
  title={A user’s guide to discrete Morse theory},
  author={Forman, Robin},
  journal={S{\'e}minaire Lotharingien de Combinatoire},
  volume={48},
  pages={B48c},
  year={2002}
}

@article{mischaikow2013morse,
  title={Morse theory for filtrations and efficient computation of persistent homology},
  author={Mischaikow, Konstantin and Nanda, Vidit},
  journal={Discrete \& Computational Geometry},
  volume={50},
  number={2},
  pages={330--353},
  year={2013},
  publisher={Springer}
}

@article{henselman,
   title = {Matroid Filtrations and Computational Persistent Homology},
   author = {Gregory Henselman-Petrusek and Robert Ghrist},
   journal = {arXiv:1606.00199},
   year = {2016},
}

@article{lampret,
  title = {Chain complex reduction via fast digraph traversal},
  author = {Leon Lampret},
  journal = {arXiv:1903.00783},
  year = {2019},
}

@book{spanier,
  title = {Algebraic Topology},
  author = {Edwin Spanier},
  publisher = {McGraw-Hill},
  year = {1966},
}

@article{joswig,
 author = {Michael Joswig and Marc E Pfetsch},
 title = {Computing optimal {M}orse matchings},
 journal = {SIAM Journal on Discrete Mathematics},
 volume = {20},
 number = {1},
 pages = {11--25},
 year = {2006},
}

@article{bauer,
 author = {Ulrich Bauer and Abhishek Rathod},
 title = {Hardness of approximation for {M}orse matching},
 journal = {SODA '19: Proceedings of the Thirtieth Annual ACM-SIAM Symposium on Discrete Algorithms},
 year = {2019},
 pages = {2663--2674},
}

@article{schulte2019multivariate,
  title={Multivariate second order Poincar{\'e} inequalities for Poisson functionals},
  author={Schulte, Matthias and Yukich, Joseph E},
  journal={Electronic journal of probability},
  volume={24},
  pages={1--42},
  year={2019},
  publisher={Institute of Mathematical Statistics and Bernoulli Society}
}

@article{kasprzak2022vector,
  title={Vector-valued statistics of binomial processes: Berry-Esseen bounds in the convex distance},
  author={Kasprzak, Miko{\l}aj J and Peccati, Giovanni},
  journal={arXiv preprint arXiv:2203.13137},
  year={2022}
}

@book{chen2011normal,
  title={Normal Approximation by Stein's Method},
  author={Chen, Louis HY and Goldstein, Larry and Shao, Qi-Man},
  year={2011},
  publisher={Springer}
}

@article{reinert2010random,
  title={Random subgraph counts and U-statistics: multivariate normal approximation via exchangeable pairs and embedding},
  author={Reinert, Gesine and R{\"o}llin, Adrian},
  journal={Journal of Applied Probability},
  volume={47},
  number={2},
  pages={378--393},
  year={2010},
  publisher={Cambridge University Press}
}

@article{kaur2020higher,
  title={Higher-order fluctuations in dense random graph models},
  author={Kaur, Gursharn and R{\"o}llin, Adrian},
  journal={Electronic Journal of Probability},
  volume={26},
  pages={1--36},
  year={2021},
  publisher={Institute of Mathematical Statistics and Bernoulli Society}
}

@article{barbour1989central,
  title={A central limit theorem for decomposable random variables with applications to random graphs},
  author={Barbour, Andrew D and Karo{\'n}ski, Michal and Ruci{\'n}ski, Andrzej},
  journal={Journal of Combinatorial Theory, Series B},
  volume={47},
  number={2},
  pages={125--145},
  year={1989},
  publisher={Elsevier}
}

@article{fang2016multivariate,
  title={A multivariate CLT for bounded decomposable random vectors with the best known rate},
  author={Fang, Xiao},
  journal={Journal of Theoretical Probability},
  volume={29},
  number={4},
  pages={1510--1523},
  year={2016},
  publisher={Springer}
}

@article{raivc2004multivariate,
  title={A multivariate CLT for decomposable random vectors with finite second moments},
  author={Rai{\v{c}}, Martin},
  journal={Journal of Theoretical Probability},
  volume={17},
  number={3},
  pages={573--603},
  year={2004},
  publisher={Springer}
}

@article{gan2017dirichlet,
  title={Dirichlet approximation of equilibrium distributions in Cannings models with mutation},
  author={Gan, Han L and R{\"o}llin, Adrian and Ross, Nathan},
  journal={Advances in Applied Probability},
  volume={49},
  number={3},
  pages={927--959},
  year={2017},
  publisher={Cambridge University Press}
}

@article{bentkus2003dependence,
  title={On the dependence of the Berry--Esseen bound on dimension},
  author={Bentkus, Vidmantas},
  journal={Journal of Statistical Planning and Inference},
  volume={113},
  number={2},
  pages={385--402},
  year={2003},
  publisher={Elsevier}
}

@incollection{meckes2009stein,
  title={On Stein’s method for multivariate normal approximation},
  author={Meckes, Elizabeth},
  booktitle={High dimensional probability V: the Luminy volume},
  pages={153--178},
  year={2009},
  publisher={Institute of Mathematical Statistics}
}

@article{erdos1959random,
  title={On random graphs I},
  author={Erd{\H{o}}s, Paul and R{\'e}nyi, Alfr{\'e}d},
  journal={Publ. math. debrecen},
  volume={6},
  number={290-297},
  pages={18},
  year={1959}
}

@article{janson1991asymptotic,
  title={The asymptotic distributions of generalized U-statistics with applications to random graphs},
  author={Janson, Svante and Nowicki, Krzysztof},
  journal={Probability Theory and Related Fields},
  volume={90},
  number={3},
  pages={341--375},
  year={1991},
  publisher={Springer}
}

@article{privault2020normal,
  title={Normal approximation for sums of weighted $ U $-statistics--application to Kolmogorov bounds in random subgraph counting},
  author={Privault, Nicolas and Serafin, Grzegorz},
  journal={Bernoulli},
  volume={26},
  number={1},
  pages={587--615},
  year={2020},
  publisher={Bernoulli Society for Mathematical Statistics and Probability}
}

@article{eichelsbacher2021kolmogorov,
  title={Kolmogorov bounds for decomposable random variables and subgraph counting by the Stein-Tikhomirov method},
  author={Eichelsbacher, Peter and Redno{\ss}, Benedikt},
  journal={arXiv:2107. 03775},
  year={2021}
}

@article{kahle2011random,
  title={Random geometric complexes},
  author={Kahle, Matthew},
  journal={Discrete \& Computational Geometry},
  volume={45},
  number={3},
  pages={553--573},
  year={2011},
  publisher={Springer}
}

@article{kahle2009topology,
  title={Topology of random clique complexes},
  author={Kahle, Matthew},
  journal={Discrete mathematics},
  volume={309},
  number={6},
  pages={1658--1671},
  year={2009},
  publisher={Elsevier}
}

@article{kahle2014sharp,
  title={Sharp vanishing thresholds for cohomology of random flag complexes},
  author={Kahle, Matthew},
  journal={Annals of Mathematics},
  pages={1085--1107},
  year={2014},
  publisher={JSTOR}
}

@article{linial2006homological,
  title={Homological connectivity of random 2-complexes},
  author={Linial, Nathan and Meshulam, Roy},
  journal={Combinatorica},
  volume={26},
  number={4},
  pages={475--487},
  year={2006},
  publisher={Springer}
}

@article{costa2016large,
  title={Large random simplicial complexes, I},
  author={Costa, A and Farber, M},
  journal={Journal of Topology and Analysis},
  volume={8},
  number={03},
  pages={399--429},
  year={2016},
  publisher={World Scientific}
}

@incollection{bobrowski2022random,
  title={Random simplicial complexes: models and phenomena},
  author={Bobrowski, Omer and Krioukov, Dmitri},
  booktitle={Higher-Order Systems},
  pages={59--96},
  year={2022},
  publisher={Springer}
}

@article{bauer2021parameterized,
  title={Parameterized inapproximability of Morse matching},
  author={Bauer, Ulrich and Rathod, Abhishek},
  journal={arXiv: 2109.04529v1},
  year={2021}
}

@article{bobrowski2018topology,
  title={Topology of random geometric complexes: a survey},
  author={Bobrowski, Omer and Kahle, Matthew},
  journal={Journal of Applied and Computational Topology},
  volume={1},
  number={3},
  pages={331--364},
  year={2018},
  publisher={Springer}
}

@book{lee2019u,
  title={$U$-statistics: Theory and Practice},
  author={Lee, A J},
  year={1990},
  publisher={Taylor \& Francis}
}

@book{korolyuk2013theory,
  title={Theory of U-statistics},
  author={Korolyuk, Vladimir S and Borovskich, Yu V},
  volume={273},
  year={2013},
  publisher={Springer Science \& Business Media}
}

@article{chatterjee2007multivariate,
  title={Multivariate normal approximation using exchangeable pairs},
  author={Chatterjee, Sourav and Meckes, Elizabeth},
  journal={Alea},
  volume={4},
  pages={257--283},
  year={2008}
}

@article{barbour1990stein,
  title={Stein's method for diffusion approximations},
  author={Barbour, Andrew D},
  journal={Probability Theory and Related Fields},
  volume={84},
  number={3},
  pages={297--322},
  year={1990},
  publisher={Springer}
}

@inproceedings{moginley1975dissociated,
  title={Dissociated random variables},
  author={McGinley, WG and Sibson, Robin},
  booktitle={Mathematical Proceedings of the Cambridge Philosophical Society},
  volume={77},
  number={1},
  pages={185--188},
  year={1975},
  organization={Cambridge University Press}
}

@article{gaunt2022bounding,
  title={Bounding Kolmogorov distances through Wasserstein and related integral probability metrics},
  author={Gaunt, Robert E and Li, Siqi},
  journal={arXiv preprint arXiv:2201.12087},
  year={2022}
}

@article{strat1,
	author = {Vidit Nanda},
	title = {Local Cohomology and Stratification},
	journal = {Foundations of Computational Mathematics},
	volume = {20},
	pages = {195--222},
	year = {2020},
}

@article{strat2,
title = {Algorithmic canonical stratifications of simplicial complexes},
journal = {Journal of Pure and Applied Algebra},
volume = {226},
number = {9},
pages = {107051},
year = {2022},
author = {Ryo Asai and Jay Shah},
}

@article{crackle,
title = {Crackle: the homology of noise},
author = {Robert J Adler and Omer Bobrowski and Shmuel Weinberger},
journal = {Discrete and Computational Geometry},
year = {2014},
pages = {680--704},
}
	
	\begin{appendix}
		\section{Proofs of Lemmas \ref{lemma:var_crit_lexi} and \ref{lemma:lower_bound_var_crit_lexi}}
		\begin{proof}[Proof of Lemma \ref{lemma:var_crit_lexi}]
			For $s \in C_{k+1}$ recall that $s_{-} = s \setminus \set{\min(s)}$. We write:
			\begin{align*}
			Y^+_s &= \prod_{i=1}^{\min(s) - 1}\left(1 - \prod_{j \in s}Y_{i,j}\right),
			&Y_s^{-} &= \prod_{i=1}^{\min(s) - 1}\left(1 - \prod_{j \in s_{-}}Y_{i,j}\right),\\
			Z_s &=  \prod_{i \neq j \in s} Y_{i,j},
			&Y_s &= Y_s^{+} - Y_s^{-}.
			\end{align*}
			Then $Z_s$ and $Y_s$ are independent and $T_{k+1} = \sum_{s \in C_{k+1}} Z_s Y_s$. Consider the variance: 
			\begin{align}\label{varexapansion}
			&\Var(T_{k+1}) = \sum_{s \in C_{k+1}} \Var(Z_sY_s) + \sum_{\substack{s \neq t \in C_{k+1} \\ \min(s) \neq \min(t)}} \Cov(Z_sY_s, Z_tY_t) \nonumber \\
			&+ \sum_{\substack{s \neq t \in C_{k+1} \\ \min(s) = \min(t)}} \Cov(Z_sY_s, Z_tY_t). 
			\end{align}
			For the first term in \eqref{varexapansion}, writing $$ \PP(Z_sY_s=1)= \mu(i) \coloneqq p^{\binom{k+1}{2}} ((1-p^{k+1})^{i-1} - (1-p^k)^{i-1})$$ we see:
			\begin{align*}
			&\sum_{s \in C_{k+1}} \Var(Z_sY_s) = \sum_{i=1}^n \sum_{\substack{s \in C_{k+1} \\ \min(s) = i}} \left( \EEbracket{(Z_sY_s)^2} - \EEbracket{Z_sY_s}^2 \right) \\
			&= \sum_{i=1}^{n-k}  \binom{n-i}{k} \left( \PP(Z_sY_s=1) - \PP(Z_sY_s=1)^2 \right) \\
			&= \sum_{i=1}^{n-k} \binom{n-i}{k} \left\{ \mu(i) - \mu(i)^2 \right\} = p^{\binom{k+1}{2}}V_{4}.
			\end{align*}

			Now consider the covariance terms in \eqref{varexapansion}, the expansion of the variance. Note that for any $s, t \in C_{k+1}$ if $s \cap t = \varnothing$, then the variables $Z_s Y_s$ and $Z_t Y_t$ can be written as functions of two disjoint sets of independent edge indicators and hence have zero covariance.
			
			Fix $s, t \in C_{k+1}$ and assume $|s \cap t| = m$ where $1 \leq m \leq k$. Note that because $m \neq k + 1$, we have $s \neq t$.
			There are $2\binom{k+1}{2} - \binom{m}{2}$ distinct edges in $s$ and $t$ combined and hence $\PP(Z_s Z_t = 1) = p^{2\binom{k+1}{2} - \binom{m}{2}}$. Also, $Y_sY_t = Y_s^+Y_t^+ + Y_s^-Y_t^- - Y_s^+Y_t^- - Y_s^-Y_t^+$. For the rest of the proof when calculating probabilities we  assume w.l.o.g.\,that $\min(s) \leq \min(t)$. Then we have for $Y_s^+ Y_t^+$:
			
			\begin{align*}
			Y_s^+Y_t^+ &= \prod_{i=1}^{\min(t) - 1}(1 - \prod_{j \in t}Y_{i,j}) \prod_{i=1}^{\min(s) - 1}(1 - \prod_{j \in s}Y_{i,j}) \\
			&=\prod_{i=1}^{\min(s) - 1}(1-\prod_{j \in s} Y_{i,j})(1-\prod_{j \in t} Y_{i,j}) \prod_{\substack{i=\min(s)\\i \in s}}^{\min(t) - 1}(1 - \prod_{j \in t}Y_{i,j}) \prod_{\substack{i=\min(s)\\ i \notin s}}^{\min(t) - 1}(1 - \prod_{j \in t}Y_{i,j}).
			\end{align*}
			
			Fix $i \in [\min(s)-1]$. Then with $\neg$ denoting the complement
			\begin{align*}
			\PP [ (1-\prod_{j \in s} Y_{i,j})&(1-\prod_{j \in t} Y_{i,j}) = 1] = \PP [ \neg(\prod_{j \in s} Y_{i,j} = 1 \cup \prod_{j \in t} Y_{i,j} = 1 ) ]  \\
			&=1-\left\{ \PP(\prod_{j \in s} Y_{i,j} = 1) + \PP(\prod_{j \in t} Y_{i,j} = 1) - \PP(\prod_{j \in s} Y_{i,j} = 1 \cap \prod_{j \in t} Y_{i,j} = 1) \right\} \\
			&=1-(2p^{k+1} - p^{2k+2-m}).
			\end{align*}
			
			Moreover, $\prod_{i=1}^{\min(s) - 1}(1-\prod_{j \in s} Y_{i,j})(1-\prod_{j \in t} Y_{i,j})$ and $\prod_{\substack{i=\min(s)\\ i \notin s}}^{\min(t) - 1}(1 - \prod_{j \in t}Y_{i,j}) $ are independent of  $Z_s Z_t$. 
			
			Recall the notation $[a,b] = \set{a, a+1, \ldots, b}$ for two positive integers $a \leq b$. Setting $q_{s, t} \coloneqq |s \cap [\min(s), \min(t) - 1]|$, 
			\begin{align*}
			&\PP\left(Y_s^+ Y_t^+ = 1 | Z_s Z_t = 1\right) \\ 
			= &\PP\left(\prod_{i=1}^{\min(s) - 1}(1-\prod_{j \in s} Y_{i,j})(1-\prod_{j \in t} Y_{i,j}) =1 \right) \PP\left(\prod_{\substack{i=\min(s)\\ i \notin s}}^{\min(t) - 1}(1 - \prod_{j \in t}Y_{i,j}) =1 \right)\\
			&\PP\left(\prod_{\substack{i=\min(s)\\ i \in s}}^{\min(t) - 1}(1 - \prod_{j \in t}Y_{i,j}) =1 \Bigg| \prod_{i \neq j \in s} Y_{i,j}\prod_{i \neq j \in t} Y_{i,j} = 1 \right)\\
			=&(1-2p^{k+1} + p^{2k+2-m})^{\min(s) - 1}(1-p^{k+1})^{\min(t) - \min(s)-q_{s, t}} (1-p^{k+1-m})^{q_{s, t}}.
			\end{align*}
			
			This strategy of splitting the product $Y_s^+Y_t^+$ into three products of independent variables, only one of which is dependent on $Z_sZ_t$ works exactly in the same way for the variables $Y_s^-Y_t^+$, $Y_s^+Y_t^-$, $Y_s^-Y_t^-$. We write $i = \min(s)$, $j = \min(t)$, and $q$ instead of $q_{s,t}$. Also, we set \begin{align*} \pi(i,j,a, b, d_1, d_2, q) \coloneqq  &
			(1-p^{a} - p^{b} + p^{a+b-d_1})^{i-1}(1-p^{a})^{j-i-q}(1-p^{a-d_2})^q.
			\end{align*} Using the described strategy we get:

			\begin{align*}
			&\PP\left(Y_s^- Y_t^- = 1 | Z_s Z_t = 1\right) = \pi(i,j,k, k, |s_{-} \cap t_{-}|, |s_{-} \cap t_{-}|, q)\\
			&\PP\left(Y_s^+ Y_t^- = 1 | Z_s Z_t = 1\right) = \pi(i,j,k, k+1, |s \cap t_{-}|, |s \cap t_{-}|, q)\\
			&\PP\left(Y_s^- Y_t^+ = 1 | Z_s Z_t = 1\right) = \pi(i,j,k+1,k,|s_{-} \cap t|,m,q).
			\end{align*}
			
			Now we are ready to calculate the covariance:
			\begin{align*}
			&\Cov(Z_sY_s, Z_tY_t) = \EEbracket{Z_sZ_tY_s^+Y_t^+} + \EEbracket{Z_sZ_tY_s^-Y_t^-} - \EEbracket{Z_sZ_tY_s^+Y_t^-} \\
			&- \EEbracket{Z_sZ_tY_s^-Y_t^+} - \EEbracket{Z_sY_s}\EEbracket{Z_tY_t}\\
			&= \PP(Z_s Z_t = 1)\Big\{\PP\left(Y^+_s Y^+_t = 1 | Z_s Z_t = 1\right) + \PP\left(Y_s^- Y_t^- = 1 | Z_s Z_t = 1\right) \\
			&- \PP\left(Y^+_s Y^-_t = 1 | Z_s Z_t = 1\right) - \PP\left(Y^-_s Y^+_t = 1 | Z_s Z_t = 1\right) \Big\} - \PP(Z_s Y_s = 1)\PP(Z_t Y_t = 1)\\ 
			&= p^{2\binom{k+1}{2} - \binom{m}{2}} (\pi(i,j,k+1, k+1, m, m, q) + \pi(i,j,k, k, |s_{-} \cap t_{-}|, |s_{-} \cap t_{-}|, q) \\
			&- \pi(i,j,k, k+1, |s \cap t_{-}|, |s \cap t_{-}|, q) - \pi(i,j,k+1,k,|s_{-} \cap t|,m,q)) -\mu(i)\mu(j).
			\end{align*}
			
			Next we consider the two covariance sums in \eqref{varexapansion} separately. First let us assume that $\min(s) \ne \min (t)$. Given $i,j \in [n-k]$, $m \in [k]$, and $q \in [\min(k+1, |j-i|)]$ define the set $$\Gamma_{k+1}(i,j,m,q) = \setc{(s,t)}{s,t \in C_{k+1}, \min(s) = i, \min(t) = j, |s \cap t| = m, \max(q_{s,t}, q_{t,s}) = q}$$ as well as\[\Gamma^+_{k+1}(i,j,m,q) = \setc{(s,t) \in \Gamma_{k+1}(i,j,m,q)}{\min(t) \in s}\] and \[\Gamma^-_{k+1}(i,j,m,q) = \setc{(s,t) \in \Gamma_{k+1}(i,j,m,q)}{\min(t) \notin s}.\]
			Next we argue that 
			\[
			|\Gamma^+_{k+1}(i,j,m,q)| = \binom{n-j}{2k+1-m-q} \binom{2k+1-m-q}{k} \binom{k}{m-1} \binom{j-i+1}{q-1}. 
			\]
			To see this, assume $i < j$. Note that to pick a pair $(s,t) \in \Gamma^+_{k+1}(i,j,m,q)$  with $ \min(s) = i$ and $ \min(t) = j$ we need to pick the $2k-m$ vertices in $s \cup t$. Firstly, we pick the vertices that are not included in $s \cap [\min(s), \min(t) - 1] = s \cap [i, j - 1]$.  Since $\min(s) \in s \cap [\min(s), \min(t) - 1]$, this amounts to choosing $2k - m - (q-1)$ vertices out of $n-j$. Then we decide which of the vertices that we have just picked will lie in $t$. This means we further need to choose $k$ out of $2k+1-m-q$ vertices. Then we choose $m-1$ out of $k$ vertices of $t$ to lie in $s \cap t$ (under the assumption that we already have $\min(t) \in s$). Finally, we choose the set $s \cap [\min(s), \min(t) - 1]$, which amounts to picking $q-1$ vertices out of $j - i +1$ possible choices. If any of the binomial coefficients are negative, we set them to 0. The case $j < i$ is analogous.
			
			An analogous argument shows that
			\[
			|\Gamma^-_{k+1}(i,j,m,q)| = \binom{n-j}{2k+1-m-q} \binom{2k+1-m-q}{k} \binom{k}{m} \binom{j-i+1}{q-1}. 
			\]
			
			Now using the covariance expression we have just derived, we get
			\begin{align*}
			&\sum_{\substack{s \neq t \in C_{k+1} \\ \min(s) \neq \min(t)}} \Cov(Z_sY_s, Z_tY_t)\\
			=  &\sum_{i=1}^{n-k} \sum_{j=i+1}^{n-k} \sum_{m=1}^k \sum_{q=1}^{\min(k+1, j-i)} \sum_{(s,t) \in \Gamma^+_{k+1}(i,j,m,q)} \Cov(Z_sY_s, Z_tY_t) \\
			&+ \sum_{i=1}^{n-k} \sum_{j=i+1}^{n-k} \sum_{m=1}^k \sum_{q=1}^{\min(k+1, j-i)} \sum_{(s,t) \in \Gamma^-_{k+1}(i,j,m,q)} \Cov(Z_sY_s, Z_tY_t) \\
			&+ \sum_{j=1}^{n-k} \sum_{i=j+1}^{n-k} \sum_{m=1}^k \sum_{q=1}^{\min(k+1, i-j)} \sum_{(s,t) \in \Gamma^+_{k+1}(j,i,m,q)} \Cov(Z_sY_s, Z_tY_t) \\
			&+ \sum_{j=1}^{n-k} \sum_{i=j+1}^{n-k} \sum_{m=1}^k \sum_{q=1}^{\min(k+1, i-j)} \sum_{(s,t) \in \Gamma^-_{k+1}(j,i,m,q)} \Cov(Z_sY_s, Z_tY_t)\\
			= &\sum_{i=1}^{n-k} \sum_{j=i+1}^{n-k} \sum_{m=1}^k \sum_{q=1}^{\min(k+1, j-i)} |\Gamma^+_{k+1}(i,j,m,q)| \Big\{ p^{2\binom{k+1}{2} - \binom{m}{2}} (\pi(i,j,k+1, k+1, m, m, q) \\
			&+ \pi(i,j,k, k, m-1, m-1, q) - \pi(i,j,k, k+1, m-1, m-1, q) \\
			&- \pi(i,j,k+1,k,m,m,q)) -\mu(i)\mu(j) \Big\}\\
			&+ \sum_{i=1}^{n-k} \sum_{j=i+1}^{n-k} \sum_{m=1}^k \sum_{q=1}^{\min(k+1, j-i)} |\Gamma^-_{k+1}(i,j,m,q)| \Big\{ p^{2\binom{k+1}{2} - \binom{m}{2}} (\pi(i,j,k+1, k+1, m, m, q) \\
			&+ \pi(i,j,k, k, m, m, q) - \pi(i,j,k, k+1, m,m, q)  \\
			&- \pi(i,j,k+1,k,m,m,q)) -\mu(i)\mu(j) \Big\}\\
			&+ \sum_{j=1}^{n-k} \sum_{i=j+1}^{n-k} \sum_{m=1}^k \sum_{q=1}^{\min(k+1, i-j)} |\Gamma^+_{k+1}(j,i,m,q)| \Big\{ p^{2\binom{k+1}{2} - \binom{m}{2}} (\pi(j,i,k+1, k+1, m, m, q) \\
			&+ \pi(j,i,k, k, m-1, m-1, q) - \pi(j,i,k, k+1, m-1, m-1, q) \\
			&- \pi(j,i,k+1,k,m,m,q)) -\mu(i)\mu(j) \Big\}\\
			&+ \sum_{j=1}^{n-k} \sum_{i=j+1}^{n-k} \sum_{m=1}^k \sum_{q=1}^{\min(k+1, i-j)} |\Gamma^-_{k+1}(j,i,m,q)| \Big\{ p^{2\binom{k+1}{2} - \binom{m}{2}} (\pi(j,i,k+1, k+1, m, m, q) \\
			&+ \pi(j,i,k, k, m, m, q) - \pi(j,i,k, k+1, m,m, q)\\
			&- \pi(j,i,k+1,k,m,m,q)) -\mu(i)\mu(j) \Big\}\\
			&= 2p^{2\binom{k+1}{2}}V_1 + 2p^{2\binom{k+1}{2}}V_2. 
			\end{align*}

			Similarly, we calculate the remaining term in the expansion of the variance \eqref{varexapansion}.
			We notice that if $i=j$, then $q=0$ and we have $\Gamma_{k+1}(i,i,m,0) = \Gamma^+_{k+1}(i,i,m,0)$. Hence,  $|\Gamma_{k+1}(i,i,m,0)| = \binom{n-i}{2k+1-m} \binom{2k+1-m}{k} \binom{k}{m-1}$, and 
			\begin{align*}
			&\sum_{\substack{s \neq t \in C_{k+1} \\ \min(s) = \min(t)}} \Cov(Z_sY_s, Z_tY_t) = \sum_{i=1}^{n-k} \sum_{m=1}^k \sum_{(s,t) \in \Gamma_{k+1}(i,i,m,0)} \Cov(Z_sY_s, Z_tY_t) \\
			= &\sum_{i=1}^{n-k} \sum_{m=1}^k  \binom{n-i}{2k+1-m} \binom{2k+1-m}{k} \binom{k}{m-1}  \\
			&\Big\{ p^{- \binom{m}{2}} \Big[(1-2p^{k+1} + p^{2k+2-m})^{i-1} + (1-2p^{k} + p^{2k+1-m})^{i-1} \\
			&- 2(1 -p^k -p^{k+1} + p^{2k+2-m})^{i-1} \Big] -  ((1-p^{k+1})^{i-1} - (1-p^k)^{i-1})^{2}\Big\} \\
			= & p^{2\binom{k+1}{2}}V_3.
			\end{align*}
		\end{proof}
		
		\begin{proof}[Proof of Lemma \ref{lemma:lower_bound_var_crit_lexi}]
			Fix $1 \leq k \leq n-1$ and $p \in (0,1)$, and consider the variance. From Lemma \ref{lemma:var_crit_lexi} we have $\Var\{ T_{k+1} \} = 2p^{2\binom{k+1}{2}}V_1 + 2p^{2\binom{k+1}{2}}V_2 + p^{2\binom{k+1}{2}}V_3 +  p^{\binom{k+1}{2}}V_4$. 
			First we lower bound $V_1$ and $V_2$ by just the negative part of the sum:
			\begin{align*}
			V_1 \geq  &-\sum_{i < j}^{n-k} \sum_{m=1}^k \sum_{q=1}^{\min(k+1, j-i)} \binom{n-j}{2k+1-m-q} \binom{2k+1-m-q}{k} \binom{k}{m-1} \binom{j-i+1}{q-1} \\
			&\Big\{(1-p^{k+1})^{i+j-2} + (1-p^{k})^{i+j-2} \\
			&+p^{- \binom{m}{2}}(1-p^{k+1})^{j-i-q}(1-p^{k+1-m})^q(1-p^{k+1}-p^k + p^{2k+1-m})^{i-1} \\
			&+p^{- \binom{m}{2}}(1-p^{k})^{j-i-q}(1-p^{k-m})^q (1-p^{k+1}-p^k + p^{2k+2-m})^{i-1}\Big\};\\
			V_2 \geq  &-\sum_{i < j}^{n-k} \sum_{m=1}^k \sum_{q=1}^{\min(k+1, j-i)} \binom{n-j}{2k+1-m-q} \binom{2k+1-m-q}{k} \binom{k}{m} \binom{j-i+1}{q-1} \\
			&\Big\{(1-p^{k+1})^{i+j-2} + (1-p^{k})^{i+j-2} \\
			&+p^{- \binom{m}{2}}(1-p^{k+1})^{j-i-q}(1-p^{k+1-m})^q(1-p^{k+1}-p^k + p^{2k+2-m})^{i-1} \\
			&+p^{- \binom{m}{2}}(1-p^{k})^{j-i-q}(1-p^{k-m})^q (1-p^{k+1}-p^k + p^{2k+2-m})^{i-1}\Big\}.
			\end{align*}
			
			Now using that $\binom{k}{m} + \binom{k}{m-1} = \binom{k+1}{m}$ and $(1-p^{k+1}) \geq (1-p^{k-m})$ for $m \geq 0$ it is easy to see that $V_1 + V_2 \geq -4R_1 - 4R_2$, where
			\begin{align*}
			&R_1 \coloneqq \sum_{i < j}^{n-k} \sum_{m=1}^k \sum_{q=1}^{\min(k+1, j-i)} \binom{n-j}{2k+1-m-q} \binom{2k+1-m-q}{k} \binom{k+1}{m} \binom{j-i+1}{q-1} \\
			&(1-p^{k+1})^{i+j-2}; \\
			&R_2 \coloneqq \sum_{i < j}^{n-k} \sum_{m=1}^k \sum_{q=1}^{\min(k+1, j-i)} \binom{n-j}{2k+1-m-q} \binom{2k+1-m-q}{k} \binom{k+1}{m} \binom{j-i+1}{q-1} \\
			&p^{- \binom{m}{2}}(1-p^{k+1})^{j-i}(1-p^{k+1}-p^k + p^{2k+1-m})^{i-1}.
			\end{align*}
			
			For $V_3$ we lower bound by terms with $m = 1$ and the negative parts of the other terms:
			\begin{align*}
			V_3 &\geq \sum_{i=1}^{n-k} \binom{n-i}{2k} \binom{2k}{k}  \left\{ (1-2p^{k+1} + p^{2k+1})^{i-1} -  (1-p^{k+1})^{2i-2}\right\}\\
			&-\sum_{i=1}^{n-k} \sum_{m=2}^k  \binom{n-i}{2k+1-m} \binom{2k+1-m}{k} \binom{k}{m-1} \\
			&\Big\{ 2p^{- \binom{m}{2}} (1 -p^k -p^{k+1} + p^{2k+2-m})^{i-1} +  2(1-p^{k+1})^{2i-2}\Big\}\\
			&=R_4 - R_3;
			\end{align*}
			here we call the positive part of the lower bound $R_4$ and the negative part $R_3$. For $V_4$ we use the trivial lower bound $V_4 \geq 0$. Hence, we have:
			\[\Var(T_{k+1}) \geq p^{2\binom{k+1}{2}}(R_4 - 8R_1 - 8R_2 - R_3).\]
			
			Let us now upper bound $R_1$:
			\begin{align*}
			R_1 &\leq \sum_{i < j}^{n-k} \sum_{m=1}^k \sum_{q=1}^{\min(k+1, j-i)} \frac{(n-j)^{2k+1-m-q}}{(2k+1-m-q)!} \frac{(2k+1-m-q)^k}{k!} \frac{(k+1)^m}{m!}\frac{(j-i+1)^{q-1}}{(q-1)!} \\
			&(1-p^{k+1})^{i+j-2} \\
			\leq& \sum_{i < j}^{n-k} \sum_{q=1}^{\min(k+1, j-i)} k \frac{n^{2k-q}}{1} \frac{(2k-1)^k}{k!} \frac{(k+1)^k}{1} \frac{n^{q-1}}{1} (1-p^{k+1})^{i+j-2}\\
			\leq & (k+1)^{k+1} \frac{n^{2k-1}}{(k-1)!} (2k-1)^k \sum_{i < j}^\infty (1-p^{k+1})^{i+j-2} \\
			= &\frac{n^{2k-1} (2k-1)^k (k+1)^{k+1} }{(k-1)!} \frac{1-p^{k+1}}{(2-p^{k+1})p^{2k+2}}.
			\end{align*}
			
			Noting that $(1-p^{k+1})^{j-i}(1-p^{k+1}-p^k + p^{2k+1-m})^{i-1} \leq (1-p^{k+1})^{j-1}$, we can bound $R_2$ in an identical way:
			\begin{align*}
			R_2 \leq &\frac{n^{2k-1} (2k-1)^k (k+1)^{k+1} }{(k-1)!} p^{-\binom{k}{2}} \sum_{i < j}^\infty (1-p^{k+1})^{j-1} \\
			&= \frac{n^{2k-1} (2k-1)^k (k+1)^{k+1} }{(k-1)!} p^{-\binom{k}{2}} \frac{1-p^{k+1}}{p^{2k+2}}.
			\end{align*}
			
			Noting that $(1 -p^k -p^{k+1} + p^{2k+2-m})^{i-1} \leq (1-p^{k+1})^{i-1}$ and $(1-p^{k+1})^{2i-2} \leq (1-p^{k+1})^{i-1}$ we proceed to bound $R_3$:
			\begin{align*}
			R_3 &\leq \sum_{i=1}^{n-k} \sum_{m=2}^k  \binom{n-i}{2k+1-m} \binom{2k+1-m}{k} \binom{k}{m-1} 2(p^{- \binom{m}{2}} + 1)(1-p^{k+1})^{i-1}\\
			&\leq \sum_{i=1}^{n-k} \sum_{m=2}^k \frac{n^{2k+1-m} (2k+1-m)^k k^{m-1}}{k!} 2(p^{- \binom{k}{2}} + 1)(1-p^{k+1})^{i-1} \\
			&\leq \sum_{i=1}^{n-k} k\frac{n^{2k+1-2} (2k+1-2)^k k^{k-1}}{k!} 2(p^{- \binom{k}{2}} + 1)(1-p^{k+1})^{i-1}\\
			&\leq \frac{n^{2k-1} (2k-1)^k k^{k}}{k!} 2(p^{- \binom{k}{2}} + 1)\sum_{i=1}^{\infty}(1-p^{k+1})^{i-1}\\
			&=\frac{n^{2k-1} (2k-1)^k k^{k}}{k!} 2(p^{- \binom{k}{2}} + 1)p^{-k-1}.
			\end{align*}
			To lower bound $R_4$ we just take the $i = 2$ term:
			\begin{align*}
			R_4 &\geq \binom{n-2}{2k} \binom{2k}{k}  \left\{ (1-2p^{k+1} + p^{2k+1})  -  (1-2p^{k+1} + p^{2k+2})\right\} \\
			& \geq \frac{(n-2)^{2k}}{(2k)^{2k}} \binom{2k}{k} p^{2k+1}(1-p).
			\end{align*}
			
			Since $R_1$, $R_2$, $R_3$ are all at most of the order $n^{2k-1}$ and $R_2$ is at least of the order $n^{2k}$, we have that for any fixed $k \geq 1$ and $p \in (0,1)$ there exists a constant $C_{p,k} > 0$ independent of $n$ and a natural number $N_{p,k}$ such that for any $n \geq N_{p,k}$:
			\[\Var(T_{k+1}) \geq p^{2\binom{k+1}{2}}(R_4 - 8R_1 - 8R_2 - R_3) \geq C_{p,k}n^{2k}.\]
		\end{proof}
	\end{appendix}
	
\end{document}